\documentclass[10pt]{amsart}

\usepackage{amssymb}
\usepackage[leqno]{amsmath}
\usepackage{bm,mathtools}
\usepackage{mathrsfs}
\usepackage{stmaryrd}
\usepackage{chemarrow}
\usepackage[norelsize]{algorithm2e}
\usepackage{enumerate}
\usepackage{graphicx}
\usepackage[pdftex,bookmarksnumbered,bookmarksopen,colorlinks,linkcolor=blue,anchorcolor=black,citecolor=blue,urlcolor=blue]{hyperref}
\usepackage[all]{xy}
\usepackage{tikz}
\usetikzlibrary{arrows}

\usepackage{multirow}
\usepackage[hyperpageref]{backref}
\usepackage{subfigure}
\usepackage{array}

\usepackage{cleveref}
\usepackage{booktabs}

  \newcounter{mnote}
  \let\oldmarginpar\marginpar
    \renewcommand\marginpar[1]{\-\oldmarginpar[\raggedleft\footnotesize #1]%
    {\raggedright\footnotesize #1}}

\newtheorem{theorem}{Theorem}[section]
\newtheorem{lemma}[theorem]{Lemma}

\newtheorem{example}[theorem]{Example}

\newtheorem{remark}[theorem]{Remark}

\newtheorem{assumption}{Assumption}[section]

\newcommand{\dd}{\,{\rm d}}

\newcommand{\curl}{\operatorname{curl}}
\renewcommand{\div}{\operatorname{div}}

\DeclareMathOperator*{\rot}{rot}

\numberwithin{equation}{section}

\begin{document}
\title[A robust mixed finite element method]{A robust lower order mixed finite element method
	for a strain gradient elasticity model}
\author{Mingqing Chen}%
\address{School of Mathematical Sciences, and MOE-LSC, Shanghai Jiao Tong University, Shanghai 200240, China}%
 \email{mingqing\_chen@sjtu.edu.cn}%
\author{Jianguo Huang}%
\address{School of Mathematical Sciences, and MOE-LSC, Shanghai Jiao Tong University, Shanghai 200240, China}%
 \email{jghuang@sjtu.edu.cn}%
 \author{Xuehai Huang}%
 \address{School of Mathematics, Shanghai University of Finance and Economics, Shanghai 200433, China}%
 \email{huang.xuehai@sufe.edu.cn}%

\thanks{The work of J. Huang was partially supported by NSFC (Grant No.\ 12071289) and the Fundamental Research Funds for the Central Universities. The work of X. Huang was partially supported by NSFC (Grant Nos.\ 12171300, 12071289), and the Natural Science Foundation of Shanghai (Grant No.\ 21ZR1480500).}

\makeatletter
\@namedef{subjclassname@2020}{\textup{2020} Mathematics Subject Classification}
\makeatother

\begin{abstract}
A robust nonconforming mixed finite element method is developed for a strain gradient elasticity (SGE) model. In two and three dimensional cases, a lower order $C^0$-continuous $H^2$-nonconforming finite element is constructed for the displacement field through enriching the quadratic Lagrange element with bubble functions. This together with the linear Lagrange element is exploited to discretize a mixed formulation of the SGE model. The robust discrete inf-sup condition is established. The sharp and uniform error estimates with respect to both the small size parameter and the Lam\'{e} coefficient are achieved, which is also verified by numerical results.  In addition, the uniform regularity of the SGE model is derived under two reasonable assumptions.
\end{abstract}


\maketitle


\section{Introduction}

Size effect of microstructures at the nanoscale has been observed in many experiments. 
In order to account for such phenomenon, higher-order continuum theory has begun to emerge even since the early twentieth century. Unlike the classic continuum theory, the corresponding constitutive relations contain additional parameters to characterize the size of materials, giving rise to various elastic mathematical models (cf. \cite{Cosserat1909,Toupin1962Elastic,MindlinTiersten1962Effects,Koiter1964couplestress,
	Mindlin1964Micro,Aifantis1984microstructural,AltanAifantis1992,RuAifantis1993simple}). Historically, Cosserat brothers propose in their celebrated work \cite{Cosserat1909} a continuum model concerning the effect of couple-stresses. Later on,  Mindlin \cite{Mindlin1964Micro} develops a more general linear elasticity theory, but the associated model contains quite a number of size parameters. To make the balance between effective prediction and computation cost, Aifantis and his collaborators introduce in \cite{AltanAifantis1992,RuAifantis1993simple} only one size parameter to produce a strain gradient elasticity (SGE) model for material deformation at the micro/nano scale. This SGE model is well accepted in engineering areas to deal with elastic-plastic problems.
We refer the reader to the survey papers \cite{AskesAifantis2011Gradient,Castrenze2015SGE} for details along this line. It is worth noting that for a microarchitecture or microstructure, if its separation of scales is not fully valid, such as pantographic \cite{RahaliGiorgio2015,AlibertDella2015}, lattice \cite{KhakaloNiiranen2018Form,YangTimofeev2020Effective} and cellular \cite{KhakaloNiiranen2020} metamaterials, one requires to invoke higher-order theory with many size parameters for mechanical computation. Moreover, homogenization methods are actively used to determine generalized material parameters (cf.  \cite{YvonnetAuffray2020homogenization,HosseiniNiiranen2022}).



Let $\Omega\subset \mathbb R^d$ ($d=2,\,3$) be a bounded domain with Lipschitz boundary. The SGE model given in \cite{AltanAifantis1992,RuAifantis1993simple} can be described as follows.
\begin{equation}
\begin{cases}\label{eq:straingradproblem}
-\div\big((\bm I-\iota^2\Delta)\bm\sigma(\bm u)\big)=\bm f & \textrm{ in }  \Omega,\\
\bm u=\partial_{\bm n}\bm u=\bm0 & \textrm{ on } \partial\Omega,
\end{cases}
\end{equation}
where $\bm n$ is the unit outward normal to $\partial\Omega$, $\bm u=(u_1, \ldots, u_d)^{\intercal}$ is the displacement field, $\bm f$ is the applied force, $\bm\sigma(\bm u)=2\mu\bm\varepsilon(\bm u)+\lambda(\div\bm u)\bm I$, and $\bm\varepsilon(\bm u)=(\varepsilon_{ij}(\bm u))_{d\times d}$ is the strain tensor field with
$
\varepsilon_{ij}(\bm u)=1/2(\partial_i u_j+\partial_ju_i).
$
Here $\lambda$ and $\mu$ are the Lam\'e constants, and $\bm I$ is the identity tensor field. The resulting stress field is given by $\tilde{\bm\sigma}(\bm u)=(\bm I-\iota^2\Delta)\bm\sigma(\bm u)$, where $\iota\in (0,1)$ denotes the size parameter of the material under discussion.
The SGE model \eqref{eq:straingradproblem} is a fourth order elliptic singular perturbation problem with small parameter $\iota$, which will reduce to a second-order linear elasticity model \eqref{eq:elasticity} when $\iota=0$. That means, the solution of the reduced problem would not meet one of the boundary conditions of problem~\eqref{eq:straingradproblem}, namely the value of the normal derivative on the boundary. This would lead to the phenomenon of boundary layer.
Moreover, as $\lambda$ goes to infinity, one can see from \eqref{eq:straingradregularity2} that $\|\div\bm u\|_2\to 0$. In other words, with $\lambda$ becoming very large, the elastic material becomes nearly incompressible, and this would deteriorate the approximation accuracy of the lower-order conforming finite element methods and lead to the locking phenomenon. Hence, it is a very challenging issue to develop a robust numerical method with respect to both size parameter $\iota$ and Lam\'e constant $\lambda$.

Finite element method (FEM) is a well-known numerical method for solving elastic problems. For example, many $C^1$-conforming FEMs are used for solving strain gradient elasticity problems in two or three dimensions (cf. \cite{AkarapuZbib2006,Zervos-Papanastasiou-Vardoulakis2001,ZervosPapanicolopulosVardoulakis2009, ManzariYonten2013SGEFEM,PapanicolopulosZervosVardoulakis2009SGEFEM,SongZhaoHe2014SGEFEM,TorabiAnsariDarvizeh2018}). For avoiding higher order shape functions so as to alleviate the computational cost, many non-conforming FEMs are developed accordingly (cf. \cite{SohChen2004SGEnoncom,TorabiAnsariDarvizeh2019SGEnoncon,ZhaoChenLo2011SGEnoncon,LiMingShi2017robust,LiMingShi2018New,LiaoMing2019,LiMingWang2021Korn}).
Another way is to use the mixed method, and in this case the displacement and stress fields can be approximated simultaneously (cf. \cite{AmanatidouAravas2002Mixed,BorstPamin1996,ShuKingFleck1999mixedfem,PhunpengBaiz2015Mixed}).

More recently, the isogeometric analysis method has also been used for solving higher-order strain gradient models (cf. \cite{BalobanovNiiranen2018,HosseiniNiiranen2022,KHAKALOLaukkanen2022,MakvandiReiherBertram2018Isogeometric,NiiranenKhakaloBalobanov2016isogeometric}). As far as we know, the method is first proposed by Hughes and his collaborators for solving various engineering problems (see \cite{HughesCottrellBazilevs2005Isogeometric,CottrellHughes2009}), and the systematic error analysis is developed by Beir\~{a}o da Veiga et al. in \cite{BeiraoBuffa2014Mathematical}. We mention that the optimal error estimates are obtained in some of the previous papers \cite{BalobanovNiiranen2018,HosseiniNiiranen2022,NiiranenKhakaloBalobanov2016isogeometric}, but the parameter dependence is not considered theoretically there.


 Now, let us review some existing results for error analysis of the SGE model~\eqref{eq:straingradproblem}.  In \cite{LiMingShi2017robust,LiMingShi2018New,LiaoMing2019,LiMingWang2021Korn}, Ming and his collaborators analyze several robust non-conforming FEMs with respect to the size parameter $\iota$ under some specific assumptions, but not robust with respect to the Lam\'{e} constant $\lambda$. To fix this problem, as the Lam\'{e} system for the linear elasticity, they introduce a ``pressure field" $p=\lambda \div \bm u$ to reformulate \eqref{eq:straingradproblem} in the form \cite{LiaoMingXu2021taylorhood}
\begin{equation}\label{straingradLameproblem}
\begin{cases}
-\div\big((\bm I-\iota^2\Delta)(2\mu\bm\varepsilon(\bm u))\big) -\nabla\big((\bm I-\iota^2\Delta)p\big)=\bm f & \textrm{ in } \Omega,\\
\div\bm u -\dfrac{1}{\lambda}p=0 & \textrm{ in } \Omega,\\
\bm u=\partial_{\bm{n}}\bm u=\bm 0 & \textrm{ on } \partial\Omega.
\end{cases}
\end{equation}
And they propose lower order nonconforming finite element methods based on \eqref{straingradLameproblem},  and derived the robust error estimate of $\bm{u}_0-\bm{u}_h$ and $p_0-p_h$ in energy norm, rather than $\bm{u}-\bm{u}_h$ and $p-p_h$, with respect to both Lam\'{e} constant $\lambda$ and size parameter $\iota$ under the assumption $\iota$ is much smaller than the mesh size $h$. Here $(\bm{u}_0, p_0)$ is the solution of problem \eqref{straingradLameproblem} with $\iota=0$, i.e. the Lam\'e system.
It is rather involved to acquire robust error estimate of $\bm{u}-\bm{u}_h$ and $p-p_h$ with respect to both Lam\'{e} constant $\lambda$ and size parameter $\iota$, which motivates this paper. Besides, based on some assumptions of weak continuity, in his PhD thesis \cite{Tian2021NonFEM} supervised by Prof. Jun Hu, Tian constructed a family of robust nonconforming finite element methods with the reduced integration technique for the primal formulation of SGE model \eqref{eq:straingradproblem} in two dimensions.

In this paper we shall devise a robust nonconforming mixed finite element method based on \eqref{straingradLameproblem} with respect to both the size parameter $\iota$ and the Lam\'{e} constant $\lambda$. The well-posedness of problem \eqref{straingradLameproblem} is related to the surjectivity of operator $\div: H_0^1(\Omega;\mathbb R^d)\cap \iota H_0^2(\Omega;\mathbb R^d)\to L_0^2(\Omega)\cap\iota H_0^1(\Omega)$, which is equivalent to the inf-sup condition
\begin{equation}\label{intro:infsup}	
\|q\|_0+\iota|q|_1\lesssim \sup_{\bm{v}\in H_0^2(\Omega;\mathbb R^d)}\frac{(\div\bm v, q)+\iota^2(\nabla\div\bm v, \nabla q)}{|\bm{v}|_1+\iota|\bm{v}|_2}\quad\forall~q\in H_0^1(\Omega)\cap L_0^2(\Omega).
\end{equation}
Here, the symbol $H_0^1(\Omega;\mathbb R^d)\cap \iota H_0^2(\Omega;\mathbb R^d)$ denotes the space $H_0^2(\Omega;\mathbb R^d)$ equipped with norm $|\bm{v}|_{1}+\iota|\bm{v}|_{2}$, while $L_0^2(\Omega)\cap\iota H_0^1(\Omega)$ denotes the space $H_0^1(\Omega)\cap L_0^2(\Omega)$ equipped with norm $\|q\|_{0}+\iota|q|_{1}$.
To develop robust finite element methods for problem~\eqref{straingradLameproblem}, we need to find a pair of finite element spaces $V_h$ and $Q_h$ satisfying the discrete analogue of the inf-sup condition~\eqref{intro:infsup}, where $V_h$ and $Q_h$ are used to discretize $H_0^2(\Omega;\mathbb R^d)$ and $H_0^1(\Omega)\cap L_0^2(\Omega)$, respectively.
The smooth finite element de Rham complexes in \cite{HuLinWu2021,ChenHuang2022b,ChenHuang2022a} will ensure the discrete inf-sup condition, but suffer from large number of degrees of freedom (DoFs) and supersmooth DoFs.
To this end, we construct a lower order $C^0$-continuous $H^2$-nonconforming finite element in two and three dimensions for the displacement field $\bm{u}$, and use the continuous linear element to discretize the pressure field $p$. The DoFs of the $H^2$-nonconforming finite element involve
\begin{equation*}	
(\div\bm v, q)_F  \quad \forall~ q\in\mathbb P_{0}(F), F\in \mathcal F(K),
\end{equation*}
which is vital to prove the robust discrete inf-sup condition with respect to the size parameter $\iota$. With finite element spaces $V_h$ and $Q_h$, we propose a robust nonconforming finite element method for problem \eqref{straingradLameproblem}.
After establishing the discrete inf-sup condition and some interpolation error estimates, we achieve the robust error estimate $O(h^{1/2})$ of $\bm{u}-\bm{u}_h$ and $p-p_h$ with respect to both the size parameter $\iota$ and the Lam\'{e} constant $\lambda$.

Another contribution of this work is building up the regularity of problem~\eqref{eq:straingradproblem}
\begin{equation*}	
|\bm u-\bm u_0|_{1} +\iota\|\bm u\|_2  + \iota^{2}\|\bm u\|_{3}+\lambda\|\div(\bm u-\bm u_0)\|_0 + \lambda\iota|\div\bm u|_1 + \lambda\iota^2\|\div\bm u\|_2 \lesssim \iota^{1/2}\|\bm f\|_0,
\end{equation*}
under two reasonable assumptions. These assumptions are interesting problems in the field of partial differential equations. We prove the first assumption in two dimensions with the aid of the regularity of the triharmonic equation.

The rest of this paper is organized as follows.
In Sect.~\ref{Sect:Notation}, we recall some notations and existing results.
Then the regularity result for SGE model is shown in Sect.~\ref{Sect:regularity}.
In Sect.~\ref{Sect:mixedform}, we construct a lower order $H^2$-nonconforming finite element, and propose a mixed finite element method for SGE model.
In Sect.~\ref{Sect:Error}, we derive the robust error estimates with respect to Lam\'e parameter $\lambda$ and size parameter $\iota$. Finally, numerical examples are provided in Sect.~\ref{Sect:test}.

\section{Preliminaries}\label{Sect:Notation}
\subsection{Notation and some basic inequalities}
Throughout this paper, let $\Omega\subset \mathbb{R}^d$ ($d=2,\,3$) be a convex polytope, which is partitioned into a family of shape regular simplices $\mathcal{T}_h=\{K\}$; $h_K={\rm diam}(K)$ and $h=\max_{K\in \mathcal{T}_h}h_K$. Let $\mathcal V(\mathcal T_h)$ and $\mathcal V^i(\mathcal T_h)$ be the sets of all and interior vertices of $\mathcal T_h$, respectively; denote by $\mathcal{E}(\mathcal T_h)$ and $\mathcal{E}^i(\mathcal T_h)$ (resp. $\mathcal{F}(\mathcal T_h)$ and $\mathcal{F}^i(\mathcal T_h)$) the sets of all and interior one-dimensional edges (resp. $(d-1)$-dimensional faces) of $\mathcal T_h$, respectively. When $d=2$, it is evident that $\mathcal{F}(\mathcal T_h)=\mathcal{E}(\mathcal T_h)$ and $\mathcal{F}^i(\mathcal T_h)=\mathcal{E}^i(\mathcal T_h)$. We next introduce the following macro elements for later requirement. For vertex $\delta\in \mathcal V(\mathcal T_h)$, edge $e\in \mathcal{E}(\mathcal T_h)$ and face $F \in \mathcal{F}(\mathcal T_h)$, write $\omega_{\delta}$, $\omega_{e}$ and $\omega_{F}$  to be respectively the unions of all simplices in $\mathcal{T}_{\delta}$, $\mathcal{T}_{e}$ and $\mathcal{T}_{F}$, where $\mathcal{T}_{\delta}$, $\mathcal{T}_{e}$ and $\mathcal{T}_{F}$ is the set of all simplices in $\mathcal{T}_h$ sharing common vertex $\delta$, edge $e$ and face $F$ respectively. In addition, for a finite set $A$, denote by $\# A$ its cardinality. For all $F \in \mathcal{F}^i(\mathcal T_h)$, which are shared by two simplices $K^+$ and $K^-$ in $\mathcal{T}_{F}$, denote by $\bm n$ the unit outward normal to $K^+$ and define the jump on $F$ as
$
\llbracket\partial_{\bm n}\bm v\rrbracket |_F \coloneqq \partial_{\bm n}\bm v|_{K_+} - \partial_{\bm n}\bm v|_{K_-}.
$
We also write $\llbracket\partial_{\bm n}\bm v\rrbracket |_F \coloneqq \partial_{\bm n}\bm v|_F$ for all $F\in\mathcal F^{\partial}(\mathcal T_h):=\mathcal F(\mathcal T_h)\backslash\mathcal F^i(\mathcal T_h)$.

Given a bounded domain $D\subset \mathbb{R}^{d}$
and integers $m,k\geq0$, denote by $H^m(D)$ the standard Sobolev space on $D$
with norm $\|\cdot\|_{m,D}$ and semi-norm $|\cdot|_{m,D}$, and $H_0^m(D)$ the closure of $C_0^{\infty}(D)$ with respect to $\|\cdot\|_{m,D}$. We abbreviate $\|\cdot\|_{m,\Omega}$ and $|\cdot|_{m,\Omega}$ as $\|\cdot\|_{m}$ and $|\cdot|_{m}$ respectively for $D=\Omega$.
The notation $(\cdot,\,\cdot)_{D}$ symbolizes the standard $L^2$ inner product on $D$, and the subscript will be omitted when $D=\Omega$.
Let $\mathbb{P}_{k}(D)$ be the set of all polynomials on $D$ with the total degree up to $k$. For a natural number $s$, set  $H^m(D; \mathbb R^s) :=H^m(D)\otimes\mathbb R^s$, $H_0^m(D; \mathbb R^s) :=H_0^m(D)\otimes\mathbb R^s$ and $\mathbb{P}_k(D; \mathbb R^s) :=\mathbb{P}_k(D)\otimes\mathbb R^s$. Let $L^2_0(D)$ be the space of functions in $L^2(D)$ with vanishing integral average values.

For a scalar function $w$ and a vector valued function $\bm v = (v_1,v_2)^{\intercal}$ in two dimensions, introduce two usual differential operations
$\curl w = (\partial_2 w, - \partial_1 w)^{\intercal}$ and $\rot \bm v = \partial_1 v_2 - \partial_2 v_1.
$

As usual, we use $\lesssim$ to represent $\leq C$, where $C$ may be a generic positive constant independent of the mesh size $h$, the material parameter $\iota$ and the Lam\'e coefficient $\lambda$. And $a\eqsim b$ indicates $a\lesssim b\lesssim a$. Assume $\Omega$ is a star-shaped domain \cite{Brenner2008FEMtheory}.
To deal with strain tensors, we need the first Korn's inequality
\cite[Corollary 11.2.25]{Brenner2008FEMtheory}
\begin{equation}\label{eq:korninequality1st}
\|\nabla\bm v\|_{0} \lesssim \|\bm\varepsilon(\bm v)\|_{0} \quad\forall~\bm v\in H^1_0(\Omega;\mathbb{R}^d),
\end{equation}
and
the $H^2$-Korn's inequality \cite[Theorem 1]{LiMingWang2021Korn}
\begin{equation}\label{eq:korninequalityH2}
( 1-1/\sqrt{2})|\bm v|_2^2\leq \|\nabla\bm\varepsilon(\bm v)\|_0^2 \quad\forall~\bm v\in H^1_0(\Omega;\mathbb{R}^d)\cap H^2(\Omega;\mathbb{R}^d).
\end{equation}


Recall the continuity of the right inverse of the divergence operator in \cite{CostabelMcIntosh2010}.
\begin{lemma}\label{lm:divrightinverse}
For $p\in H_0^{m}(\Omega)\cap L_0^2(\Omega)$ with non-negative integer $m$, there exists $\bm{v}\in H_0^{m+1}(\Omega;\mathbb R^d)$ such that $\div\bm{v}=p$ and
\begin{equation*}
\|\nabla\bm v\|_{j}\lesssim \|p\|_{j} \quad {\rm  for }\;  j=0,1,\ldots, m.
\end{equation*}
\end{lemma}

\subsection{Mixed formulation}
In order to solve the problem \eqref{eq:straingradproblem} with a mixed finite element method, we recall its mixed formulation in \cite{LiaoMingXu2021taylorhood}.
The variational formulation of \eqref{straingradLameproblem} is to find $\bm u\in V:=H_0^2(\Omega; \mathbb R^d)$ and $p\in Q:=H_0^1(\Omega)\cap L^2_0(\Omega)$ such that
\begin{equation}\label{eq:straingradmixedform}
\begin{cases}
a(\bm u, \bm v)
+b(\bm v, p)=(\bm f, \bm v)&  \forall~\bm v\in V,\\
b(\bm u, q)-c(p,q)=0 & \forall~q\in Q,
\end{cases}
\end{equation}
where
\begin{align*}
a(\bm u, \bm v)&:=2\mu(\bm\varepsilon(\bm u), \bm\varepsilon(\bm v))_{\iota}=2\mu\big((\bm\varepsilon(\bm u), \bm\varepsilon(\bm v))+\iota^2(\nabla\bm\varepsilon(\bm u), \nabla\bm\varepsilon(\bm v))\big),	 \\
b(\bm v, p)&:=(\div\bm v, p)_{\iota}=(\div\bm v, p)+\iota^2(\nabla\div\bm v, \nabla p), \\
c(p,q)&:=(p, q)_{\iota}/\lambda=\big((p, q)+\iota^2(\nabla p, \nabla q)\big)/\lambda
\end{align*}
with a weighted $H^1$ inner product $(p, q)_{\iota}:=(p, q)+\iota^2(\nabla p, \nabla q)$.
To make the forthcoming discussion in a compact way, we further introduce the following weighted Sobolev norms over $V$ and $Q$:
\[
\|\bm v\|_V : = \left(|\bm v|_1^2 + \iota ^2 |\bm v|_2^2\right)^{1/2}\quad \forall\; \bm v\in V; \quad \|q\|_Q : = \left(\|q\|_0^2 + \iota ^2 |q|_1^2\right)^{1/2}\quad \forall\; q\in Q.
\]

Now, it is easy to check that the linear forms $a(\cdot,\cdot)$, $b(\cdot,\cdot)$ and $c(\cdot,\cdot)$ are bounded on $V\times V$, $V\times Q$ and $Q\times Q$, respectively. The coercivity of the linear form $a(\cdot,\cdot)$ on $V\times V$ results from Korn's inequalities~\eqref{eq:korninequality1st}-\eqref{eq:korninequalityH2}. And the inf-sup condition
\begin{equation*}	
\|q\|_{Q}\lesssim  \sup\limits_{\bm v \in V}\frac {b(\bm v,q)} {\|\bm v\|_V } \quad \forall~q\in Q
\end{equation*}
holds from Lemma~\ref{lm:divrightinverse}. Applying the Babu{\v{s}}ka-Brezzi theory \cite{BoffiBrezziFortin2013}, we have the stability
\begin{equation*}	
\|\bm u\|_V+\|p\|_Q \lesssim \sup_{\bm v\in V,\, q\in Q}\frac{a(\bm u, \bm v)+b(\bm v, p)-b(\bm u, q)+c(p,q)}{\|\bm v\|_V+\|q\|_Q} \quad\forall~\bm u\in V, p\in Q.
\end{equation*}
Therefore, the mixed formulation \eqref{eq:straingradmixedform} is well-posed.


\section{Some regularity estimates}\label{Sect:regularity}

In order to derive the robust error estimates for the proposed mixed finite element method in the next section, we require to develop a series of regularity estimates for the underlying (system of) partial differential equations. In fact, these results are interesting themselves.

\begin{lemma}\label{lm:fourthorderStokes}
	If $\bm f \in H^{-2}(\Omega; \mathbb{R}^d)$, then the following problem
	\begin{equation}\label{eq:fourthorderStokes}
	\begin{cases}
	\Delta^2\bm\phi + \nabla\Delta p=\bm f & {\rm in }\; \Omega,\\
	\Delta\div\bm\phi= 0 & {\rm in }\; \Omega,\\
	\bm\phi=\partial_{\bm{n}}\bm\phi=\bm 0 & {\rm on }\; \partial\Omega,
	\end{cases}
	\end{equation}
	has a unique solution $(\bm\phi, p)\in H_0^2(\Omega; \mathbb R^d)\times(H_0^1(\Omega)\cap L_0^2(\Omega))$ such that
	\begin{equation}\label{eq:fourthorderStokesregularity0}
	\|\bm\phi\|_2+\|p\|_1\lesssim \|\bm f\|_{-2}.
	\end{equation}
\end{lemma}
\begin{proof}
	It is routine that the weak formulation of problem \eqref{eq:fourthorderStokes} is to find $(\bm\phi, p)\in H_0^2(\Omega; \mathbb R^d)\times(H_0^1(\Omega)\cap L_0^2(\Omega))$ such that
	\begin{equation}\label{eq:biharmonicstokes}
	\begin{cases}
	(\nabla^2\bm\phi, \nabla^2\bm\psi)
	+(\nabla\div\bm\psi, \nabla p)=\langle\bm f, \bm\psi\rangle&  \forall~\bm\psi\in H_0^2(\Omega; \mathbb R^d),\\
	(\nabla\div\bm\phi, \nabla q)=0 & \forall~q\in H_0^1(\Omega)\cap L_0^2(\Omega).
	\end{cases}
	\end{equation}
	Clearly, the bilinear form $(\nabla^2\cdot, \nabla^2\cdot)$ is  uniformly coercive over $H_0^2(\Omega; \mathbb R^d)$. On the other hand, thanks to Lemma~\ref{lm:divrightinverse}, for every $q\in H_0^1(\Omega)\cap L_0^2(\Omega)$, there exists a $\bm \chi\in H_0^2(\Omega; \mathbb R^d)$ such that $q=\div\bm\chi$ which admits the estimate $\|\nabla\bm\chi\|_1\lesssim \|q\|_1$. Hence,
\begin{align*}
	\|q\|_1&\lesssim (\nabla q, \nabla q)/\|\nabla q\|_0\lesssim (\nabla\div\bm\chi,\nabla q)/\|\nabla\bm\chi\|_1 \lesssim \sup_{\bm\psi\in H_0^2(\Omega; \mathbb R^d)}\frac{(\nabla\div\bm\psi, \nabla q)}{\|\bm{\psi}\|_2}.
\end{align*}
	Thus, we conclude the desired result from the Babu{\v{s}}ka-Brezzi theory \cite{BoffiBrezziFortin2013}.
\end{proof}

Next, we need to introduce two assumptions on the two auxiliary problems, which are the basis to develop subtle estimates for problem \eqref{eq:straingradproblem}.
\begin{assumption}\label{RegularityAssumption1}
	Assume that $\bm f \in H^{-1}(\Omega; \mathbb{R}^d)$. Let $(\bm\phi, p)\in H_0^2(\Omega; \mathbb R^d)\times(H_0^1(\Omega)\cap L_0^2(\Omega))$ be the solution of problem~\eqref{eq:fourthorderStokes}. Then $\bm{\phi}\in H^3(\Omega; \mathbb R^d)$ and $p\in H^2(\Omega)$ which admit the estimate
	\begin{equation}\label{eq:fourthorderStokesregularity1}
	\|\bm\phi\|_3+\|p\|_2\lesssim \|\bm f\|_{-1}.
	\end{equation}
\end{assumption}

We will prove Assumption~\ref{RegularityAssumption1} holds for a convex polygon in Section~\ref{sec:assumption12d}. Since the solution $\bm{\phi}$ of \eqref{eq:biharmonicstokes} is zero when $\bm{f}\in\nabla L^2(\Omega)$,
by replacing $\bm{f}$ in problem \eqref{eq:fourthorderStokes} with $\bm{f}+\nabla q$, we have from \eqref{eq:fourthorderStokesregularity1} that
\begin{equation}\label{eq:fourthorderStokesregularity2}
\|\bm\phi\|_3\lesssim
\inf_{q\in L^2(\Omega)}\|\bm f+\nabla q\|_{-1}.
\end{equation}

\begin{assumption}\label{RegularityAssumption2}
	Assume that $\bm f\in H^{-1}(\Omega; \mathbb R^d)$. Let $\bm u\in H_0^2(\Omega; \mathbb R^d)$ be the solution of the following problem:
	\begin{equation*}
	\begin{cases}
	\Delta (\mathcal{L} \bm u)=\bm f & {\rm  in }\; \Omega,\\
	\bm u=\partial_{\bm{n}}  \bm u= \bm 0 & {\rm on }\; \partial\Omega,
	\end{cases}
	\end{equation*}
	where $\mathcal{L} \bm u = \mu \Delta \bm u + (\lambda + \mu)\nabla (\div \bm u)$ is the usual Lam\'{e} operator, with $\lambda \in[0,\Lambda]$ and $\mu\in [\mu_0,\mu_1]$ being the Lam\'{e} constants. Then $\bm u\in H^3(\Omega; \mathbb R^d)$ and it admits the estimate
	\[
	\| \bm u\|_3 \lesssim \|\bm f\|_{-1},
	\]
	where the hidden constant may depend on $\Lambda$, $\mu_0$ and $\mu_1$.
\end{assumption}
\begin{remark}
	If $\partial \Omega \in C^{\infty}$, Assumption \ref{RegularityAssumption2} holds in terms of the mathematical theory in \cite{AgmonDouglisNirenberg1964}. However, if $\Omega$ is a polytope domain, as shown in \cite{KozlovMazyaRossmann2001spectral}, one should first figure out the operator pencil of the underlying problem and then discover the spectrum distribution in order to determine under which conditions Assumption \ref{RegularityAssumption2} holds. Such study is rather involved, and we cannot ensure this assumption holds even for convex polygons until now.
\end{remark}

\subsection{Regularity results for triharmonic equations in two dimensions and applications}\label{sec:assumption12d}
In this subsection, in view of the mathematical theory in \cite{Dauge1988elliptic,KozlovMazyaRossmann2001spectral}, we are going to develop the regularity theory for triharmonic equations with homogeneous boundary conditions over convex polygons. As far as we know, this result is new. It has its interest itself and will also be used to show  Assumption~\ref{RegularityAssumption1} holds for a convex polygon in combination with some results in \cite{CostabelMcIntosh2010}.

Before introducing and proving our regularity estimates, let us recall some existing results in \cite{Dauge1988elliptic}. As shown in Fig. \ref{corner1}, let $\Omega$ be a convex polygon with $X$ as a generic vertex which has an interior angle $\alpha\in (0,\pi)$. Let $L$ be a strongly elliptic $2m$-order differential operator defined on $\overline{\Omega}$ with constant coefficients, where $m$ is any given positive integer. Consider the following homogeneous Dirichlet boundary problem 
\begin{equation}\label{eq:2m_dirichlet}
\begin{cases}
Lu = f &\text{in } \Omega,\\
u = \partial_{\bm n} u = \cdots = \partial_{\bm n}^{m-1}u =0 &\text{on } \partial \Omega,
\end{cases}
\end{equation}
which induces the following operators for each positive real number $s$:
\begin{align*}
L^{(s)} : H_0^m(\Omega)\cap H^{s+m}(\Omega) &\,\, \rightarrow\,\, H^{s-m}(\Omega) ,\\
u{\,\,\,\,\,\,\,\,\,\,\,\,\,\,\,\,\,\,\,\,\,\,\,} &\,\, \rightarrow\,\, L u.
\end{align*}
A key problem is to determine the regularity property of $L^{(s)}$ for $s>0$.

\begin{figure}[htbp]
	\centering
	\includegraphics[width=4.5cm]{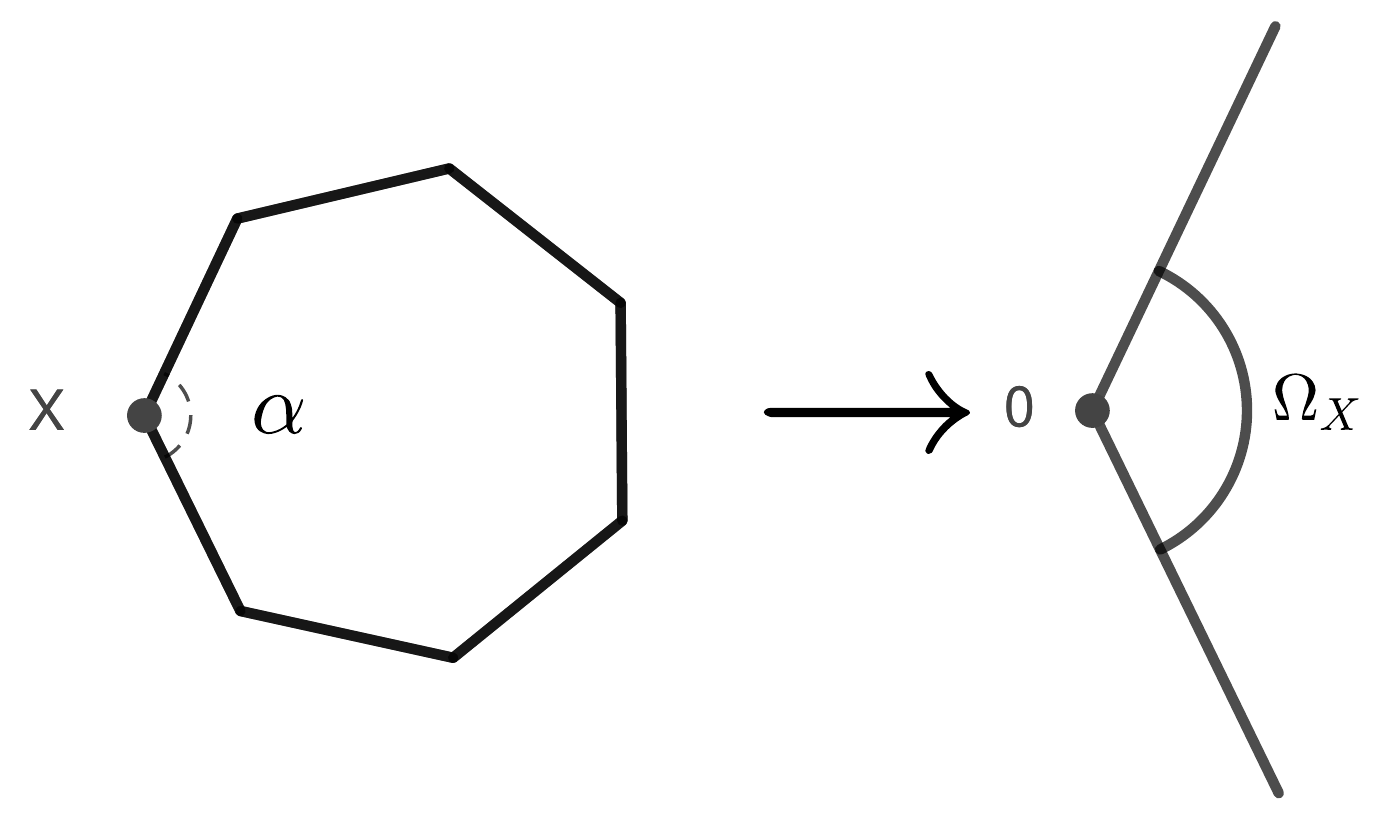}
	\caption{The schema of a polygonal vertex.}
	\label{corner1}
\end{figure}
As a matter of fact, the problem can be answered locally. As shown in Fig. \ref{corner1}, for any given vertex $X$, we might assume the unit circle $S_X$ centered at $X$ has intersection with each of two edges of $\Omega$ which share the vertex $X$, and then write $\Omega_X=\Omega\cap S_X$. Let $L_X$ be the principal part of $L$. We introduce a polar coordinate system $(r,\theta)$ with $X$ as the origin. Then the differential operator $L_X$ can be expressed as
\[
L_X(D_{\bm x}) = r^{-2m}\mathcal{L}_X(\theta;r\partial_r , D_{\theta}),
\]
where $D_{\bm x} = \frac 1 {{\rm i}}\partial_{\bm x}$, $D_{\theta} = \frac 1 {{\rm i}} \partial_{\theta}$. Next, we replace $r\partial_r$ by a complex variable $\lambda\in \mathbb{C}$, so as to get a holomorphic operator $\mathcal{L}_X(\lambda) =\mathcal{L}_X(\theta;\lambda , D_{\theta})$  which acts from $H^m_0(\Omega_X)\cap H^{s+m}(\Omega_X)$ into $H^{s-m}(\Omega_X)$. $\mathcal{L}_X(\lambda)$ is simply written as $\mathcal{L}(\lambda)$ when there is no confusion caused.

As shown in \cite{Dauge1988elliptic}, the angle singularities of solution of elliptic equations have very close relationship with spectrum properties of $\mathcal{L}(\lambda)$. Recall that a point $\lambda_0\in\mathbb{C}$ is said to be regular if the operator $\mathcal{L}(\lambda_0)$ is invertible. The set of all non-regular points is called the spectrum of the operator. Then,  by the well-known Peetre's theorem and the main theorem in \cite[p. 10]{Dauge1988elliptic} we can have a simplified result from \cite{Dauge1988elliptic}, described as follows.
\begin{theorem}[Dauge]\label{thm:Dauge88}
	Let $\Omega$ be a convex polygon and let $L$ be a strongly elliptic $2m$-order differential operator defined on $\overline{\Omega}$ with constant coefficients. Assume that the boundary value problem \eqref{eq:2m_dirichlet} has a unique solution $u\in H^m_0(\Omega)$. Moreover, $s\geq 0$ and $s\not = \{\frac 12, \frac 32,\cdots, m-\frac12\}$. If any $\lambda \in \mathbb{C}$ with ${\rm Re}\lambda\in [m-1,s+m-1]$ is not spectrum of $\mathcal{L}(\lambda)$ and $f\in H^{s-m}(\Omega)$, then $u\in H^{s+m}(\Omega)$, and it holds the estimate
	\begin{equation*}
	\|u\|_{s+m}\lesssim \|f\|_{s-m}.
	\end{equation*}
\end{theorem}

On the other hand, it is rather involved to figure out the closed-form of $\mathcal{L}(\lambda)$ and analyze its spectrum distribution. However, it is very lucky that for the differential operator $L = (-\Delta)^m$, we have from \cite[Chapter 8]{KozlovMazyaRossmann2001spectral} that
\[
\mathcal{L}(\lambda) = (-1)^{m}\prod\limits _{j=0}^{m-1}(\partial_{\theta}^2 + (\lambda - 2j)^2),
\]
and the following result holds.
\begin{lemma}\label{lm:spectrum}
	\cite[Chapter 8]{KozlovMazyaRossmann2001spectral}
	If $\Omega$ is a convex polygon, and $L = (-\Delta)^m$ for a positive integer $m$, then we have
	\begin{enumerate}
		\item the spectrum of $\mathcal{L}(\lambda)$ does not contain the points $m,\,m+1,\,\cdots\,,2m-1$;
		\item the strip $m-2\leq{\rm Re} \lambda\leq m $ does not contain any eigenvalues of $\mathcal{L}(\lambda)$.
	\end{enumerate}
\end{lemma}

%
%


Making use of the above two results, we can derive the regularity theory for the triharmonic equation given below.
\begin{theorem}\label{th:sixthorderregularity}
	Assume $\Omega$ is a convex polygon. Let $w\in H^3_0(\Omega)$ satisfy
	\begin{equation}\label{eq:sixthorder}
	-\Delta^3 w =f \in H^{-2}(\Omega).
	\end{equation}
	Then $w\in H^4(\Omega)$ and it admits the estimate
	\begin{equation*}
	\| w\|_{4} \lesssim \|f\|_{-2}.
	\end{equation*}
\end{theorem}
\begin{proof}
	The weak formulation of problem \eqref{eq:sixthorder} is to find $w\in H_0^3(\Omega)$ such that
	\[
	\int_{\Omega}\nabla^3 w:\nabla^3 v \dd x=\langle f,v\rangle\quad \forall\; v\in H_0^3(\Omega).
	\]
	Here, for any two third-order tensors $\bm A$ and $\bm B$, $\bm A: \bm B=\sum_{ijk}A_{ijk}B_{ijk}$, and $\langle\cdot,\cdot\rangle$ denotes the duality pair between $H^{-3}(\Omega)$ and $H_0^3(\Omega)$. Hence, by the Poincar\'{e} inequality and the Lax-Milgram theorem, the last problem has a unique solution $w\in H_0^3(\Omega)$. In other words, the boundary value problem \eqref{eq:sixthorder} has a unique solution $w\in H_0^3(\Omega)$ for all $f\in H^{-3}(\Omega)$. According to Lemma \ref{lm:spectrum} with $m=3$, $\mathcal{L}(\lambda)$ does not have a spectrum point in the strip $1\leq{\rm Re} \lambda\leq 3$. Hence, we can conclude the desired result from Theorem \ref{thm:Dauge88} with $s=1$.
\end{proof}

With the help of the above theorem, we can obtain the following result.

\begin{theorem}\label{thm:fourthorderStokesregularity}
	Assumption~\ref{RegularityAssumption1} holds whenever $\Omega$ is a convex polygon and $\bm f \in H^{-1}(\Omega; \mathbb{R}^2)$.
\end{theorem}
\begin{proof}
	
	{\rm Step 1}: According to the second and third equations in \eqref{eq:fourthorderStokes}, $\div\bm\phi\in H_0^1(\Omega)$ satisfies the Laplace equation, and hence $\div\bm\phi=0$.
	Since $\bm\phi \in H_0^2(\Omega; \mathbb R^2)$, there exists a scalar function $w\in H^3_0(\Omega)$ \cite{CostabelMcIntosh2010} such that $\bm \phi = \curl w$.
	Then apply the $\rot$ operator to the first equality of \eqref{eq:fourthorderStokes} to get
	\begin{equation*}		
	-\Delta^3w=\rot(\Delta^2\curl w)=\rot(\Delta^2\bm{\phi})=\rot(\Delta^2\bm{\phi}+\nabla\Delta p)=\rot\bm{f}.
	\end{equation*}
	Due to Theorem \ref{th:sixthorderregularity}, we know that $w\in H^4({\Omega})$ admits the estimate
	\[
	\|w\|_4 \lesssim \|\rot \bm f\|_{-2} \lesssim \|\bm f\|_{-1}.
	\]
	By the relation $\bm \phi = \curl w$, we immediately have $\bm \phi \in H^3(\Omega;\mathbb{R}^2)$ and
	\begin{equation} \label{eq:phi3}
	\|\bm \phi\|_3 \lesssim \|\bm f\|_{-1}.
	\end{equation}
	
	{\rm Step 2}: Rewrite the first equation in \eqref{eq:fourthorderStokes} as $\nabla \Delta p = \bm f - \Delta^2 \bm \phi.$
	Since $\bm f - \Delta^2 \bm \phi \in H^{-1}(\Omega;\mathbb{R}^2)$ and $\Delta p\in H^{-1}(\Omega)$,
	apply Theorem 2.2 in \cite{GiraultRaviart1986} to get $\Delta p\in L^2(\Omega)$ and
	\begin{equation*}
	\|\Delta p\|_0 \lesssim \|\Delta p\|_{-1}+\|\bm f - \Delta^2 \bm \phi\|_{-1}\lesssim \|p\|_{1}+\|\bm f - \Delta^2 \bm \phi\|_{-1},
	\end{equation*}
	which combined with \eqref{eq:fourthorderStokesregularity0} and \eqref{eq:phi3} yields
	\begin{equation*}
	\|\Delta p\|_0 \lesssim \|\bm f \|_{-1}.
	\end{equation*}
	Noting that $\Omega$ is convex, by the regularity of Poisson's equation,
	\begin{equation*}
	\|p\|_2 \lesssim \|\Delta p\|_0 \lesssim \|\bm f \|_{-1}.
	\end{equation*}
	Now combining \eqref{eq:phi3} and the last inequality gives \eqref{eq:fourthorderStokesregularity1}.
\end{proof}

\subsection{Regularity of the SGE model}
Now we are ready to show the regularity results of the SGE model. From now on we always suppose Assumptions~\ref{RegularityAssumption1}-\ref{RegularityAssumption2} hold.
\begin{lemma}\label{lm:fourthorderelasregularity}
	Let $\bm w\in H_0^2(\Omega; \mathbb R^d)$ be the solution of
	\begin{equation}\label{eq:fourthorderelasproblem}
	\begin{cases}
	\div\Delta\big(2\mu\bm\varepsilon(\bm w)+\lambda(\div\bm w)\bm I\big)=\bm f & {\rm  in }\; \Omega,\\
	\bm w=\partial_{\bm{n}}\bm w=\bm 0 & {\rm on }\; \partial\Omega
	\end{cases}
	\end{equation}
with $\bm f\in H^{-1}(\Omega; \mathbb{R}^d)$.
	Suppose Assumptions \ref{RegularityAssumption1}-\ref{RegularityAssumption2} hold.
	We have
	\begin{align}\label{eq:fourthorderelasregularity}
	\|\bm w\|_{2+j}+\lambda\|\div\bm w\|_{1+j}\lesssim \|\bm f\|_{j-2} \quad\textrm{ for } j=0,1,
	\\
\label{eq:fourthorderelasregularityw}
	\|\bm w\|_{3}\lesssim \inf_{q\in L^2(\Omega)}\|\bm f+\nabla q\|_{-1} + \frac{1}{\lambda}\|\bm f\|_{-1}.
	\end{align}
\end{lemma}
\begin{proof}
	We follow the argument of Lemma~2.2 in \cite{BrennerSung1992} to prove \eqref{eq:fourthorderelasregularity}.
	The weak formulation of \eqref{eq:fourthorderelasproblem} is
	\begin{equation}\label{eq:20210918}
	2\mu(\nabla\bm\varepsilon(\bm w), \nabla\bm\varepsilon(\bm v)) + \lambda(\nabla\div\bm w, \nabla\div\bm v)=(\bm f, \bm v) \quad\forall~\bm v\in H_0^2(\Omega; \mathbb R^d),
	\end{equation}
	which is well-posed by the $H^2$-Korn's inequality \eqref{eq:korninequalityH2}, and it holds
	$
	\|\bm w\|_2\lesssim \|\bm f\|_{-2}.
	$
	By Lemma~\ref{lm:divrightinverse}, there exists $\widetilde{\bm w}\in H_0^2(\Omega; \mathbb R^d)$ such that
	$\div\widetilde{\bm w}=\div\bm w$ and $\|\widetilde{\bm w}\|_2\lesssim \|\div\bm w\|_1.$
	Taking $\bm v=\widetilde{\bm w}$ in \eqref{eq:20210918}, it follows that
	\begin{equation*}
	\lambda|\div\bm w|_1^2=(\bm f, \widetilde{\bm w})-2\mu(\nabla\bm\varepsilon(\bm w), \nabla\bm\varepsilon(\widetilde{\bm w}))\lesssim \|\bm f\|_{-2}\|\widetilde{\bm w}\|_2+\|\bm w\|_2\|\widetilde{\bm w}\|_2,
	\end{equation*}
	which gives \eqref{eq:fourthorderelasregularity} for $j=0$.
	
	Next we prove \eqref{eq:fourthorderelasregularity} for $j=1$. The first equation in \eqref{eq:fourthorderelasproblem} can be rewritten as
	\begin{equation*}
	\Delta^2\bm w+\frac{\mu+\lambda}{\mu}\nabla\Delta\div\bm w=\frac{1}{\mu}\bm f  \quad \textrm{ in } \Omega.
	\end{equation*}
	Clearly $\bm w\in H^3(\Omega; \mathbb R^d)$ under Assumption~\ref{RegularityAssumption2}.
	Employing \cref{lm:divrightinverse} again, there exists $\widetilde{\bm w}\in H^3(\Omega; \mathbb R^d)\cap H_0^2(\Omega; \mathbb R^d)$ such that
	$\div\widetilde{\bm w}=\div\bm w$ and $\|\widetilde{\bm w}\|_3\lesssim \|\div\bm w\|_2.$
	With the help of $\widetilde{\bm w}$, we have
	\begin{equation*}
	\begin{cases}
	\Delta^2\bm\phi + \nabla\Delta p=\frac{1}{\mu}\bm f-\Delta^2\widetilde{\bm w} & \textrm{ in } \Omega,\\
	\Delta\div\bm\phi=0 & \textrm{ in } \Omega,\\
	\bm\phi=\partial_{\bm{n}}\bm\phi=\bm0 & \textrm{ on } \partial\Omega,
	\end{cases}
	\end{equation*}
	where $\bm\phi=\bm w-\widetilde{\bm w}$ and $p=\frac{\mu+\lambda}{\mu}\div\bm w$. By \eqref{eq:fourthorderStokesregularity1}-\eqref{eq:fourthorderStokesregularity2}, 
	\begin{align}\label{eq:20210923-1}
	\|\bm w-\widetilde{\bm w}\|_3+\frac{\mu+\lambda}{\mu}\|\div\bm w\|_2&\lesssim \|\bm f\|_{-1}+\|\widetilde{\bm w}\|_3,
	\\
\label{eq:20210923-2}
	\|\bm w-\widetilde{\bm w}\|_3&\lesssim
	\inf_{q\in L^2(\Omega)}\|\bm f+\nabla q\|_{-1} +\|\widetilde{\bm w}\|_3.
	\end{align}
	By \eqref{eq:20210923-1}, there exists a constant $C>0$ such that
	\begin{equation*}
	\|\bm w\|_3+\frac{\mu+\lambda}{\mu}\|\div\bm w\|_2\leq C(\|\bm f\|_{-1}+\|\div\bm w\|_2).
	\end{equation*}
	When $\lambda>2\mu C$, i.e. $C<\frac{\lambda}{2\mu}$, we get
	\begin{equation*}
	\|\bm w\|_3+\frac{2\mu+\lambda}{2\mu}\|\div\bm w\|_2\leq C\|\bm f\|_{-1}.
	\end{equation*}
	Thus, \eqref{eq:fourthorderelasregularity} holds for $j=1$. When $0<\lambda\leq2\mu C$, \eqref{eq:fourthorderelasregularity} for $j=1$ is guaranteed by Assumption~\ref{RegularityAssumption2}.
	
	It follows from \eqref{eq:20210923-2} that
	\begin{equation*}
	\|\bm w\|_3\lesssim
	\inf_{q\in L^2(\Omega)}\|\bm f+\nabla q\|_{-1} +\|\widetilde{\bm w}\|_3 \lesssim \inf_{q\in L^2(\Omega)}\|\bm f+\nabla q\|_{-1}+ \|\div\bm w\|_2,
	\end{equation*}
	which together with \eqref{eq:fourthorderelasregularity} yields \eqref{eq:fourthorderelasregularityw}.
\end{proof}

Taking $\iota=0$, problem \eqref{eq:straingradproblem} becomes the linear elasticity model. Let $\bm u_0\in H_0^1(\Omega; \mathbb R^d)$ be the solution of the linear elasticity problem
\begin{equation}\label{eq:elasticity}
\begin{cases}
-\mu\Delta\bm u_0 - (\lambda+\mu)\nabla\div\bm u_0 =\bm f & \textrm{ in } \Omega,\\
\bm u_0=\bm0 & \textrm{ on } \partial\Omega.
\end{cases}
\end{equation}
According to \cite{BrennerSung1992,MazyaRossmann2010}, it holds
\begin{equation}\label{eq:u0regularity}
\|\bm u_0\|_2+\lambda\|\div\bm u_0\|_1 \lesssim \|\bm f\|_0.
\end{equation}

\begin{lemma}\label{lm:uregularity}
	With the same assumptions as Lemma \ref{lm:fourthorderelasregularity},
	let $\bm u\in H_0^2(\Omega; \mathbb R^d)$ be the solution of problem \eqref{eq:straingradproblem}, and $\bm u_0\in H_0^1(\Omega; \mathbb R^d)$ satisfy the linear elasticity problem~\eqref{eq:elasticity}.
	It holds
	\begin{equation}\label{eq:straingradregularity1}
	|\bm u-\bm u_0|_{1} +\iota\|\bm u\|_2  + \iota^{2}\|\bm u\|_{3} \lesssim \iota^{1/2}\|\bm f\|_0.
	\end{equation}
\end{lemma}
\begin{proof}
	By \eqref{eq:straingradproblem} and \eqref{eq:elasticity}, we have
	\begin{equation}\label{eq:20210919}
	\div\Delta\bm\sigma(\bm u)=\iota^{-2}\div\bm\sigma(\bm u-\bm u_0).
	\end{equation}
	Then it follows from \eqref{eq:fourthorderelasregularityw} that
	\begin{align}
	\iota^{2}\|\bm u\|_{3}&\lesssim \inf_{q\in L^2(\Omega)}\|\div\bm\sigma(\bm u-\bm u_0)+\nabla q\|_{-1} + \frac{1}{\lambda}\|\div\bm\sigma(\bm u-\bm u_0)\|_{-1} \notag\\
	&\lesssim |\bm u-\bm u_0|_{1} + \frac{1}{\lambda}|\bm u-\bm u_0|_{1}+\|\div(\bm u-\bm u_0)\|_0\lesssim |\bm u-\bm u_0|_{1}.    \label{eq:20210919-1}
	\end{align}
	Multiply \eqref{eq:20210919} by $\bm u-\bm u_0\in H^2(\Omega; \mathbb R^d)\cap H_0^1(\Omega; \mathbb R^d)$ and apply the integration by parts to get
	\begin{align}
	&\quad \iota^{2}(\nabla\bm\sigma(\bm u), \nabla\bm\varepsilon(\bm u)) + (\bm\sigma(\bm u-\bm u_0), \bm\varepsilon(\bm u-\bm u_0)) \notag\\
	&=\iota^{2}(\nabla\bm\sigma(\bm u), \nabla\bm\varepsilon(\bm u_0)) - \iota^{2}(\partial_{\bm{n}}\bm\sigma(\bm u), \bm\varepsilon(\bm u_0))_{\partial\Omega}.    \label{eq:20210919-2}
	\end{align}
	By \eqref{eq:u0regularity},
	\begin{align*}
	(\nabla\bm\sigma(\bm u), \nabla\bm\varepsilon(\bm u_0)) & = 2\mu(\nabla\bm\varepsilon(\bm u), \nabla\bm\varepsilon(\bm u_0))+\lambda(\nabla\div\bm u, \nabla\div\bm u_0) \\
	& \lesssim |\bm u|_2(|\bm u_0|_2+\lambda|\div\bm u_0|_1)\lesssim |\bm u|_2\|\bm f\|_0.
	\end{align*}
	Applying the multiplicative trace inequality, \eqref{eq:u0regularity} and \eqref{eq:20210919-1},
	\begin{align*}
	-(\partial_{\bm{n}}\bm\sigma(\bm u), \bm\varepsilon(\bm u_0))_{\partial\Omega}&=-2\mu(\partial_{\bm{n}}\bm\varepsilon(\bm u), \bm\varepsilon(\bm u_0))_{\partial\Omega}-\lambda(\partial_{\bm{n}}\div\bm u, \div\bm u_0)_{\partial\Omega} \\
	&\lesssim \|\partial_{\bm{n}}\bm\varepsilon(\bm u)\|_{0,\partial\Omega} \|\bm\varepsilon(\bm u_0)\|_{0,\partial\Omega}  + \lambda\|\partial_{\bm{n}}\div\bm u\|_{0,\partial\Omega}\|\div\bm u_0\|_{0,\partial\Omega} \\
	&\lesssim \|\bm u\|_{2}^{1/2}\|\bm u\|_{3}^{1/2} (\|\bm u_0\|_{2} +\lambda\|\div\bm u_0\|_1) \lesssim \|\bm u\|_{2}^{1/2}\|\bm u\|_{3}^{1/2} \|\bm f\|_0 \\
	&\lesssim \iota^{-1}\|\bm u\|_{2}^{1/2} \|\bm f\|_0|\bm u-\bm u_0|_{1}^{1/2}.
	\end{align*}
	Combining the last two inequalities and \eqref{eq:20210919-2} yields
	\begin{align*}
	&\quad \iota^{2}(\nabla\bm\sigma(\bm u), \nabla\bm\varepsilon(\bm u)) + (\bm\sigma(\bm u-\bm u_0), \bm\varepsilon(\bm u-\bm u_0)) \\
	&\lesssim \left(\iota^{3/2}|\bm u|_2 +  \iota^{1/2}\|\bm u\|_{2}^{1/2} |\bm u-\bm u_0|_{1}^{1/2}\right)\iota^{1/2}\|\bm f\|_0 \lesssim \left(\iota\|\bm u\|_2 + |\bm u-\bm u_0|_{1}\right)\iota^{1/2}\|\bm f\|_0,
	\end{align*}
	which implies
	\begin{equation*}
	\iota\|\bm u\|_2 + |\bm u-\bm u_0|_{1}+\lambda^{1/2}\|\div(\bm u-\bm u_0)\|_0+\lambda^{1/2}\iota|\div\bm u|_1\lesssim \iota^{1/2}\|\bm f\|_0.
	\end{equation*}
	Finally, we get \eqref{eq:straingradregularity1} from \eqref{eq:20210919-1}.
\end{proof}

\begin{theorem}\label{tm:upregularity}
	Let $\bm u\in H_0^2(\Omega; \mathbb R^d)$ be the solution of problem \eqref{eq:straingradproblem}, and $\bm u_0\in H_0^1(\Omega; \mathbb R^d)$ satisfy the linear elasticity problem \eqref{eq:elasticity}. Under the same assumptions as Lemma \ref{lm:fourthorderelasregularity},
	it holds
	\begin{equation}\label{eq:straingradregularity2}
	\lambda\|\div(\bm u-\bm u_0)\|_0 + \lambda\iota|\div\bm u|_1 + \lambda\iota^2\|\div\bm u\|_2 \lesssim \iota^{1/2}\|\bm f\|_0.
	\end{equation}
\end{theorem}
\begin{proof}
	Applying Lemma~\ref{lm:fourthorderelasregularity} to \eqref{eq:20210919}, we obtain from \eqref{eq:fourthorderelasregularity} and \eqref{eq:straingradregularity1} that
	\begin{align}
	\lambda\iota^{2}\|\div\bm u\|_{2}&\lesssim \|\div\bm\sigma(\bm u-\bm u_0)\|_{-1}\lesssim \|\bm\sigma(\bm u-\bm u_0)\|_{0}\lesssim |\bm u-\bm u_0|_1
	+\lambda\|\div(\bm u-\bm u_0)\|_0 \notag\\
	&\lesssim \iota^{1/2}\|\bm f\|_0+\lambda\|\div(\bm u-\bm u_0)\|_0. \label{eq:20210923-4}
	\end{align}
	Hence it suffices to prove
	\begin{equation}\label{eq:straingradregularity0}
	\lambda\|\div(\bm u-\bm u_0)\|_0 + \lambda\iota|\div\bm u|_1 \lesssim \iota^{1/2}\|\bm f\|_0.
	\end{equation}
	Thanks to Lemma~\ref{lm:divrightinverse}, there exists $\bm v\in H^2(\Omega; \mathbb R^d)\cap H_0^1(\Omega; \mathbb R^d)$ such that
	\begin{equation}\label{eq:20210923-3}
	\div\bm v=\div(\bm u-\bm u_0),\quad \|\bm v\|_1\lesssim \|\div(\bm u-\bm u_0)\|_0.
	\end{equation}
	Multiply \eqref{eq:20210919} by $\bm v$ and apply the integration by parts to get
	\begin{align}
	&\quad \lambda\|\div(\bm u-\bm u_0)\|_0^2 +\lambda\iota^{2}|\div\bm u|_1^2 \notag\\
	&=-2\mu(\bm\varepsilon(\bm u-\bm u_0), \bm\varepsilon(\bm v)) + 2\mu\iota^{2}(\Delta\bm\varepsilon(\bm u), \bm\varepsilon(\bm v)) - \lambda\iota^{2}(\Delta\div\bm u, \div\bm u_0). \label{eq:20211217}
	\end{align}
	By the multiplicative trace inequality and \eqref{eq:20210923-4}, it holds
	\begin{align*}
	- \lambda\iota^{2}(\Delta\div\bm u, \div\bm u_0)&=\lambda\iota^{2}(\nabla\div\bm u, \nabla\div\bm u_0)- \lambda\iota^{2}(\partial_{\bm{n}}\div\bm u, \div\bm u_0)_{\partial\Omega} \\
	&\lesssim \lambda\iota^{2}|\div\bm u|_1|\div\bm u_0|_1 + \lambda\iota^{2}|\div\bm u|_1^{1/2}\|\div\bm u\|_2^{1/2}\|\div\bm u_0\|_1 \\
	&\lesssim \lambda\iota^{2}|\div\bm u|_1|\div\bm u_0|_1 + \lambda^{1/2}\iota^{5/4}|\div\bm u|_1^{1/2}\|\bm f\|_0^{1/2}\|\div\bm u_0\|_1 \\
	&\quad + \lambda\iota|\div\bm u|_1^{1/2}\|\div(\bm u-\bm u_0)\|_0^{1/2}\|\div\bm u_0\|_1.
	\end{align*}
	Then we acquire from \eqref{eq:20210923-3}-\eqref{eq:20211217} that
	\begin{align*}
	&\quad \lambda\|\div(\bm u-\bm u_0)\|_0^2 +\lambda\iota^{2}|\div\bm u|_1^2 \\
	&\lesssim (|\bm u-\bm u_0|_1 + \iota^{2}|\bm u|_3)|\bm v|_1 +\lambda\iota^{2}|\div\bm u|_1|\div\bm u_0|_1 \\
	&\quad + \lambda^{1/2}\iota^{5/4}|\div\bm u|_1^{1/2}\|\bm f\|_0^{1/2}\|\div\bm u_0\|_1+ \lambda\iota|\div\bm u|_1^{1/2}\|\div(\bm u-\bm u_0)\|_0^{1/2}\|\div\bm u_0\|_1 \\
	&\lesssim (|\bm u-\bm u_0|_1 + \iota^{2}|\bm u|_3) \|\div(\bm u-\bm u_0)\|_0 +(\lambda\iota^{2}|\div\bm u|_1^2)^{1/2}(\lambda\iota^2\|\div\bm u_0\|_1^2)^{1/2} \\
	&\quad + \iota^{3/2}|\div\bm u|_1\|\bm f\|_0 + \lambda\iota\|\div\bm u_0\|_1^2 \\
	&\quad + (\lambda\iota|\div\bm u|_1\|\div(\bm u-\bm u_0)\|_0)^{1/2}(\lambda\iota\|\div\bm u_0\|_1^2)^{1/2}.
	\end{align*}
	This implies
	\begin{align*}
	&\quad \lambda\|\div(\bm u-\bm u_0)\|_0^2 +\lambda\iota^{2}|\div\bm u|_1^2 \lesssim \frac{1}{\lambda}(|\bm u-\bm u_0|_1 + \iota^{2}|\bm u|_3)^2+ \lambda\iota\|\div\bm u_0\|_1^2 + \frac{1}{\lambda}\iota\|\bm f\|_0^2,
	\end{align*}
	i.e.,
	\begin{equation*}
	\lambda\|\div(\bm u-\bm u_0)\|_0 +\lambda\iota|\div\bm u|_1 \lesssim |\bm u-\bm u_0|_1 + \iota^{2}|\bm u|_3+\lambda\iota^{1/2}\|\div\bm u_0\|_1+\iota^{1/2}\|\bm f\|_0.
	\end{equation*}
	Therefore, we conclude \eqref{eq:straingradregularity0} from
	\eqref{eq:straingradregularity1} and \eqref{eq:u0regularity}.
\end{proof}

\section{Mixed finite element method}\label{Sect:mixedform}

In this section we will construct a lower order $C^0$-continuous $H^2$-nonconforming finite element, and apply to discretize the variational formulation \eqref{eq:straingradmixedform} of SGE model \eqref{straingradLameproblem}.

\subsection{$H^2$-nonconforming finite element}\label{Sect:H1ncFEM}
For a $d$-dimensional simplex $K$ ($d=2,\,3$) with vertices $a_1, \ldots, a_{d+1}$, let $\mathcal V(K)$, $\mathcal E(K)$ and $\mathcal F(K)$ be the sets of all vertices, one-dimensional edges and $(d-1)$-dimensional faces of $K$ respectively. Set $\mathcal V^i(K):=\mathcal V(K)\cap \mathcal V^i(\mathcal T_h)$.
For $s=1,\ldots, d+1$, denote by $\lambda_s$ the barycentric coordinate corresponding to $a_s$ and there exists $F_s\in \mathcal F(K)$ such that $\lambda_s|_{F_s}=0$. Let $b_K =
\prod_{s=1}^{d+1}\lambda_s =\lambda_jb_{F_j}$ be the bubble function of $K$, where $b_{F_j}$ is the bubble function of $F_j$ with $j=1,\cdots, d+1$.

Given a face $F\in\mathcal F(K)$ with unit normal vector $\bm n$, for a vector function $\bm v$, define the tangential component $\Pi_F\bm v=\bm v-(\bm v\cdot\bm n)\bm n$. Then
\begin{equation}\label{eq:divdecomp}
\div \bm v=\div((\bm v\cdot\bm n)\bm n+\Pi_F\bm v)=\partial_{\bm{n}}(\bm v\cdot\bm n)+\div_F\bm v,
\end{equation}
where face divergence $\div_F\bm v:=\div(\Pi_F\bm v)$.

With previous preparation, take
\begin{equation*}
V(K):=\mathbb P_{2}(K;\mathbb R^d)+b_K\mathbb P_{1}(K;\mathbb R^d)+b_K^2\mathbb P_{0}(K;\mathbb R^d)
\end{equation*}
as the space of shape functions. Then
\begin{equation*}
\dim V(K)=d{d+2\choose2}+d(d+1)+d= \frac{1}{2} d(d+2)(d+3).
\end{equation*}
The degrees of freedom (DoFs) are chosen as
\begin{align}
\bm v (\delta) & \quad\forall~\delta\in \mathcal V(K),\label{H2lowestncfemnDdof1}\\
\frac{1}{|e|}(\bm v, \bm q)_e & \quad \forall~\bm{q}\in\mathbb P_{0}(e;\mathbb R^d), e\in \mathcal E(K), \label{H2lowestncfemnDdof2}\\
\frac{1}{|F|}(\partial_{\bm{n}}(\bm v\cdot\bm t_i), q)_F & \quad \forall~ q\in\mathbb P_{0}(F), F\in \mathcal F(K), i=1,\ldots, d-1, \label{H2lowestncfemnDdof31}\\
\frac{1}{|F|}(\div\bm v, q)_F & \quad \forall~ q\in\mathbb P_{0}(F), F\in \mathcal F(K), \label{H2lowestncfemnDdof32}\\
\frac{1}{|K|}(\bm v, \bm q)_K & \quad \forall~\bm q\in\mathbb P_{0}(K;\mathbb R^d), \label{H2lowestncfemnDdof4}
\end{align}
where $\bm t_1,\ldots,\bm t_{d-1}$ are $(d-1)$ unit orthogonal tangential vectors of $F$. DoF \eqref{H2lowestncfemnDdof32} is vital to prove the robust discrete inf-sup condition with respect to the size parameter $\iota$.

\begin{lemma}
	The degrees of freedom \eqref{H2lowestncfemnDdof1}-\eqref{H2lowestncfemnDdof4} are uni-solvent for $ V(K)$.
\end{lemma}
\begin{proof}
	The number of the degrees of freedom \eqref{H2lowestncfemnDdof1}-\eqref{H2lowestncfemnDdof4} is
	\[
	d(d+1)+d{d+1\choose2}+d(d+1)+d = \frac{1}{2}d(d+2)(d+3)=\dim V(K).
	\]
	
	Take $\bm v\in V(K)$ and assume all the degrees of freedom \eqref{H2lowestncfemnDdof1}-\eqref{H2lowestncfemnDdof4} vanish, then we prove $\bm v=\bm0$.
	Noting that $\bm v|_{F}\in\mathbb P_{2}(F;\mathbb R^d)$ for each $F\in\mathcal F(K)$, we get from the vanishing DoFs \eqref{H2lowestncfemnDdof1}-\eqref{H2lowestncfemnDdof2} that $\bm v|_{\partial K}=\bm0$. As a result there exist $\bm v_1\in\mathbb P_{1}(K;\mathbb R^d)$ and $\bm v_2\in\mathbb P_{0}(K;\mathbb R^d)$ such that $\bm v=b_K\bm v_1+b_K^2\bm v_2$. Thanks to \eqref{eq:divdecomp}, the vanishing DoF \eqref{H2lowestncfemnDdof32} implies $(\partial_{\bm{n}}(\bm v\cdot\bm n), q)_F=0$ for all $q\in\mathbb P_{0}(F)$, which combined with the vanishing DoF~\eqref{H2lowestncfemnDdof31} yields $(\partial_{\bm{n}}\bm v, \bm q)_F=0$ for all $\bm q\in\mathbb P_{0}(F;\mathbb R^d)$, and thus
	\begin{equation*}
	(b_F\bm v_1, \bm q)_F=0\quad\forall~\bm q\in\mathbb P_{0}(F;\mathbb R^d), F\in \mathcal F(K).
	\end{equation*}
    As a result $\bm v_1=\bm 0$.
	Finally, we conclude $\bm v=\bm v_2=\bm0$ from the vanishing DoF \eqref{H2lowestncfemnDdof4}.
\end{proof}

Define a global $H^2$-nonconforming finite element space:
\begin{align*}
V_h=&\{\bm v_h\in L^2(\Omega;\mathbb R^d)| \, \bm v_h|_K\in V(K) \textrm{ for each } K\in\mathcal T_h, \textrm{ all the  DoFs \eqref{H2lowestncfemnDdof1}-\eqref{H2lowestncfemnDdof4}} \\
&\qquad\qquad\qquad \textrm{ are single-valued}, \textrm{and DoFs \eqref{H2lowestncfemnDdof1}-\eqref{H2lowestncfemnDdof32} on boundary vanish}\}.
\end{align*}
Clearly $ V_h\subset H_0^1(\Omega;\mathbb R^d)$, but $ V_h\not\subset H^2(\Omega;\mathbb R^d)$.
The finite element space $V_h$ has the weak continuity
\begin{equation}\label{eq:weakcontinuity}
\int_F\llbracket\nabla_h\bm v_h\rrbracket{\rm d}S=\bm0\quad \forall~\bm v_h\in V_h, F\in\mathcal F(\mathcal T_h).
\end{equation}
Here $\nabla_h$ is the elementwise version of $\nabla$ with respect to $\mathcal T_h$.

Introduce the Lagrange element space $Q_h=Q_h^L\cap Q$, where
\[
Q_h^L:=\{q\in H^1(\Omega) \,|\, q|_K\in\mathbb P_1(K) \textrm{ for each } K\in\mathcal T_h\}.
\]
Then the dicretization formulation of \eqref{eq:straingradmixedform} is to find $(\bm u_h,p_h)\in V_h\times Q_h$ such that
\begin{equation}\label{eq:discretestraingradmixedform}
\begin{cases}
a_h(\bm u_h, \bm v_h)
+b_h(\bm v_h, p_h)=(\bm f, \bm v_h)&  \forall~\bm v_h\in V_h,\\
b_h(\bm u_h, q_h)-c(p_h,q_h)=0 & \forall~q_h\in Q_h,
\end{cases}
\end{equation}
where
\[
a_h(\bm u_h, \bm v_h):=2\mu\big((\bm\varepsilon(\bm u_h), \bm\varepsilon(\bm v_h))+\iota^2(\nabla_h\bm\varepsilon(\bm u_h), \nabla_h\bm\varepsilon(\bm v_h))\big),
\]
\[
b_h(\bm v_h, p_h):=(\div\bm v_h, p_h)+\iota^2(\nabla_h\div\bm v_h, \nabla p_h).
\]
Define squared norm
$\|\bm v_h\|^2_{V,h} : = |\bm v_h|_1^2 + \iota ^2 |\bm v_h|_{2,h}^2$ with $|\bm v_h|_{2,h}^2=\sum_{K\in\mathcal T_h}|\bm v_h|_{2,K}^2.$ Clearly the discrete linear forms $a_h(\cdot,\cdot)$ and $b_h(\cdot,\cdot)$ are continuous on $V_h\times V_h$ and $V_h\times Q_h$, respectively.
Since $
( 1-1/\sqrt{2})|\bm v_h|_{2,h}^2\lesssim \|\nabla_h\bm\varepsilon(\bm v_h)\|_0^2$ and the Korn's inequality~\eqref{eq:korninequality1st}, we achieve the discrete coercivity
\begin{equation}\label{eq:discretecoercivity}
\|\bm v_h\|_{V,h}^2\lesssim a_h(\bm v_h, \bm v_h)\quad\forall~\bm{v}_h\in V_h.
\end{equation}

\subsection{ Discrete inf-sup condition and uni-solvence}
To derive the unisolvence of the mixed finite element method \eqref{eq:discretestraingradmixedform},
we first introduce the second order Brezzi-Douglas-Marini (BDM) element \cite{BoffiBrezziFortin2013,ChenHuang2022}. 
The second order BDM element takes $\mathbb P_{2}(K;\mathbb R^d)$ as the shape function space,
and the degrees of freedom are chosen as \cite{ChenHuang2022}
\begin{align}
(\bm v\cdot\bm n, q)_F & \quad \forall~ q\in\mathbb P_{2}(F), F\in \mathcal F(K), \label{BDMdof1}\\
(\bm v, \bm q)_K & \quad \forall~\bm q\in\mathbb P_{1}(K;\mathbb R^d)\; {\rm satisfying }\; \bm x\cdot\bm q\in\mathbb P_{1}(K). \label{BDMdof2}
\end{align}
Let $I_K^{BDM}: H^1(K; \mathbb R^d)\to \mathbb P_{2}(K;\mathbb R^d)$ be the nodal interpolation operator based on DoFs \eqref{BDMdof1}-\eqref{BDMdof2}. It holds
\begin{equation}\label{eq:BDMdivcd}
\div(I_K^{BDM}\bm v)=Q_K(\div\bm v) \quad\forall~\bm v\in H^1(K; \mathbb R^d),
\end{equation}
where $Q_K$ is the standard $L^2$ projection operator from $L^2(K)$ to $\mathbb P_1(K)$. Let $I_h^{BDM}$ be the global version of $I_K^{BDM}$.

Now we define interpolation operator $I_h: H_0^2(\Omega; \mathbb R^d)\to V_h$ as follows:
\begin{align}
(I_h\bm v) (\delta)&=\frac{1}{\#\mathcal T_{\delta}}\sum_{K\in\mathcal T_{\delta}}(I_K^{BDM}\bm v)(\delta),\notag\\
(I_h\bm v, \bm q)_e&=\frac{1}{\#\mathcal T_{e}}\sum_{K\in\mathcal T_{e}}(I_K^{BDM}\bm v, \bm q)_e  \quad \forall~\bm q\in\mathbb P_{0}(e;\mathbb R^d), \notag\\
(\partial_{\bm{n}}(\Pi_F(I_h\bm v)), \bm q)_F&=\frac{1}{\#\mathcal T_{F}}\sum_{K\in\mathcal T_{F}}(\partial_{\bm{n}}(\Pi_F(I_K^{BDM}\bm v)), \bm q)_F  \quad \forall~\bm q\in\mathbb P_{0}(F;\mathbb R^{d-1}), \notag\\
(\div(I_h\bm v), q)_F&=\frac{1}{\#\mathcal T_{F}}\sum_{K\in\mathcal T_{F}}(\div(I_K^{BDM}\bm v), q)_F \quad \forall~q\in\mathbb P_{0}(F), \label{eq:20210821-1}\\
(I_h\bm v, \bm q)_K&=(\bm v, \bm q)_K  \quad \forall~\bm q\in\mathbb P_{0}(K;\mathbb R^d), \label{eq:20210821-2}
\end{align}
for $\delta\in \mathcal V^i(\mathcal T_h)$, $e\in \mathcal E^i(\mathcal T_h)$, $F\in \mathcal F^i(\mathcal T_h)$ and $ K\in \mathcal T_h$.

\begin{lemma}\label{lem:Ihestimate}
	We have
	\begin{align}\label{eq:Ihestimate1}
	|I_h\bm v|_{j,h}&\lesssim |\bm v|_j\quad\forall~\bm v\in H_0^2(\Omega;\mathbb R^d), j=1,2,
	\\
\label{eq:Ihestimate3}
	\sum_{i=0}^3h_K^{i}|\bm v-I_h\bm v|_{i,K} &\lesssim h_K^3\sum_{\delta\in\mathcal V(K)}|\bm v|_{3, \omega_{\delta}}\quad\forall~\bm v\in H_0^2(\Omega;\mathbb R^d)\cap H^3(\Omega;\mathbb R^d).
	\end{align}
\end{lemma}
\begin{proof}
	Thanks to \eqref{BDMdof2} and \eqref{eq:20210821-2}, it follows that
	\begin{equation*}
	(I_h\bm v - I_K^{BDM}\bm v, \bm q)_K=0  \quad \forall~\bm{q}\in\mathbb P_{0}(K;\mathbb R^d).
	\end{equation*}
	Then
	by the inverse inequality, scaling argument and \eqref{eq:BDMdivcd},
	\begin{align*}
	&h_K^{2i}|I_h\bm v - I_K^{BDM}\bm v|_{i,K}^2 \\
	\lesssim& h_K^d\sum_{\delta\in \mathcal V(K)}(I_h\bm v - I_K^{BDM}\bm v)^2(\delta) + h_K^{d-1}\sum_{e\in \mathcal E(K)}\|I_h\bm v - I_K^{BDM}\bm v\|_{0,e}^2\\
	&+ h_K^{3}\sum_{F\in \mathcal F(K)}\|\partial_{\bm{n}}(\Pi_F(I_h\bm v-I_K^{BDM}\bm v))\|_{0,F}^2 + h_K^{3}\sum_{F\in \mathcal F(K)}\|\div(I_h\bm v-I_K^{BDM}\bm v)\|_{0,F}^2  \\
    \lesssim	& h_K\sum_{\delta\in \mathcal V(K)}\sum_{F\in \mathcal F(T_h),\delta\in F}\|\llbracket I_K^{BDM}\bm v\rrbracket\|_{0,F}^2+h_K^3\sum_{F\in \mathcal F(K)}\|\llbracket \partial_{\bm{n}}(I_K^{BDM}\bm v)\rrbracket\|_{0,F}^2 \\
	& + h_K^{3}\sum_{F\in \mathcal F(K)}\|\llbracket Q_K(\div\bm v)\rrbracket\|_{0,F}^2.
	\end{align*}
	This implies
	\begin{align*}
	h_K^{2i}|\bm v-I_h\bm v|_{i,K}^2&\lesssim h_K^{2i}|\bm v-I_K^{BDM}\bm v|_{i,K}^2 + h_K\sum_{\delta\in \mathcal V(K)}\sum_{F\in \mathcal F(T_h),\delta\in F}\|\llbracket I_K^{BDM}\bm v\rrbracket\|_{0,F}^2 \\
	&\quad +h_K^3\sum_{F\in \mathcal F(K)}\|\llbracket \partial_{\bm{n}}(I_K^{BDM}\bm v)\rrbracket\|_{0,F}^2+ h_K^{3}\sum_{F\in \mathcal F(K)}\|\llbracket Q_K(\div\bm v)\rrbracket\|_{0,F}^2.
	\end{align*}
	Finally, we end the proof by employing the inverse inequality, the trace inequality, and the error estimates of $I_K^{BDM}$ and $Q_K$.
\end{proof}


Next we present the discrete inf-sup condition.
\begin{lemma}
	It holds the discrete inf-sup condition
	\begin{equation}\label{eq:ncfeminfsup}
	\|q_h\|_{Q}=\|q_h\|_0+\iota|q_h|_1\lesssim\sup_{\bm v_h\in\bm V_h}\frac{b_h(\bm v_h, q_h)}{\|\bm v_h\|_{V,h}} \quad\forall~q_h\in Q_h.
	\end{equation}
\end{lemma}
\begin{proof}
By Lemma~\ref{lm:divrightinverse},
	there exists $\bm v\in H_0^2(\Omega;\mathbb R^d)$ satisfying $\div\bm v=q_h$ and
	$
	|\bm v|_1+\iota|\bm v|_2\lesssim \|q_h\|_0+\iota|q_h|_1.
	$
	Take $\bm v_h=I_h\bm v\in H_0^1(\Omega;\mathbb R^d)$. By \eqref{eq:Ihestimate1},
	\begin{equation}\label{eq:20210821-3}
	|\bm v_h|_1+\iota|\bm v_h|_{2,h}\lesssim |\bm v|_1+\iota|\bm v|_{2}\lesssim \|q_h\|_0+\iota|q_h|_1.
	\end{equation}
	For $K\in\mathcal T_h$, by \eqref{eq:BDMdivcd} and $\div\bm v\in Q_h$,
	we have $\div(I_K^{BDM}\bm v)=Q_K(\div\bm v)=\div\bm v$.
	Then we get from \eqref{eq:20210821-1} and \eqref{eq:20210821-2} that
	\begin{align*}
	(\div(\bm v_h-\bm v), q_h)&=-(\bm v_h-\bm v, \nabla q_h)=0,
	\\
	(\nabla_h\div(\bm v_h-\bm v), \nabla q_h)&=\sum_{K\in\mathcal T_h}(\div(\bm v_h-\bm v), \partial_{\bm{n}} q_h)_{\partial K}=0.
	\end{align*}
	Combining the last two equations gives
\begin{equation*}	
	b_h(\bm v_h, q_h)=(\div\bm v, q_h)+\iota^2(\nabla\div\bm v, \nabla q_h)=\|q_h\|_0^2+\iota^2|q_h|_1^2,
\end{equation*}
	which together with \eqref{eq:20210821-3} ends the proof.
\end{proof}

\begin{theorem}
	It holds the discrete stability
	\begin{equation}\label{eq:ncfemstability}
	\|\tilde{\bm u}_h\|_{V,h}+\|\tilde{p}_h\|_Q \lesssim \sup_{\bm v_h\in V_h\atop q_h\in Q_h}\frac{a_h(\tilde{\bm u}_h, \bm v_h)+b_h(\bm v_h, \tilde{p}_h)-b_h(\tilde{\bm u}_h, q_h)+c(\tilde{p}_h,q_h)}{\|\bm v_h\|_{V,h}+\|q_h\|_Q}
	\end{equation}
for $\tilde{\bm u}_h\in V_h$ and $\tilde{p}_h\in Q_h$.
	Then the mixed finite element method \eqref{eq:discretestraingradmixedform} is well-posed.
\end{theorem}
\begin{proof}
	Applying the Babu{\v{s}}ka-Brezzi theory \cite{BoffiBrezziFortin2013}, we get from the discrete coercivity~\eqref{eq:discretecoercivity} and the discrete inf-sup condition \eqref{eq:ncfeminfsup} that
\begin{equation*}
	\|\tilde{\bm u}_h\|_{V,h}+\|\tilde{p}_h\|_Q \lesssim \sup_{\bm v_h\in V_h, q_h\in Q_h}\frac{a_h(\tilde{\bm u}_h, \bm v_h)+b_h(\bm v_h, \tilde{p}_h)-b_h(\tilde{\bm u}_h, q_h)}{\|\bm v_h\|_{V,h}+\|q_h\|_Q}
\end{equation*}
for $\tilde{\bm u}_h\in V_h$ and $\tilde{p}_h\in Q_h$,
	which indicates \eqref{eq:ncfemstability}.
\end{proof}

\section{Error analysis}\label{Sect:Error}
We will analyze the mixed finite element method \eqref{eq:discretestraingradmixedform} in this section, and derive robust error estimates of $\|\bm u - \bm u_h\|_{V,h} + \|p- p_h\|_{Q}$ with respect to parameters $\iota$ and $\lambda$.

\subsection{Interpolation estimates}

Denote by $I^{SZ}_h: H_0^1(\Omega)\to Q_h^L\cap H_0^1(\Omega)$ the Scott-Zhang interpolation operator with the homogeneous boundary condition \cite{ScottZhang1990}. Since $I^{SZ}_hv\not\in Q_h$ for $I^{SZ}_hv\not\in L_0^2(\Omega)$, we will modify $I^{SZ}_hv$.

For $K\in\mathcal T_h$, let $N_K:=\#\mathcal V^i(K)$ be the number of interior vertices of $K$. We assume the triangulation $\mathcal T_h$ satisfies $\min_{K\in\mathcal T_h}N_K\geq1$. Define operator $\mathcal Q_h: L^2(\Omega)\to  Q_h^L\cap H_0^1(\Omega)$ by
\begin{equation*}	
(\mathcal Q_hv)(\delta)=\sum_{K\in\mathcal T_{\delta}}\frac{d+1}{N_K|\omega_{\delta}| }\int_Kv~{\rm d}x\quad\forall~\delta\in\mathcal V^i(\mathcal T_h),
\end{equation*}
where $|\omega_{\delta}|$ is the geometrical measure of $\omega_{\delta}$.

\begin{lemma}
	For $v\in L^2(\Omega)$, we have $v-\mathcal Q_hv\in L_0^2(\Omega)$, and
	\begin{equation}\label{eq:QhL02prop}
	\|\mathcal Q_hv\|_{0,K}+h_K|\mathcal Q_hv|_{1,K} \lesssim \sum_{\delta\in\mathcal V^i(K)}\|v\|_{0,\omega_{\delta}}.
	\end{equation}
\end{lemma}
\begin{proof}
	By the fact $(\mathcal Q_hv)|_K\in\mathbb P_1(K)$ for $K\in\mathcal T_h$,
	\begin{align*}
&\quad	\int_{\Omega}\mathcal Q_hv~{\rm d}x=\sum_{K\in\mathcal T_h}\int_{K}\mathcal Q_hv~{\rm d}x=\frac{1}{d+1}\sum_{K\in\mathcal T_h}\sum_{\delta\in\mathcal V^i(K)}|K|(\mathcal Q_hv)(\delta) \\
&=\frac{1}{d+1}\sum_{\delta\in\mathcal V^i(\mathcal T_h)}\sum_{K\in\mathcal T_{\delta}}|K|(\mathcal Q_hv)(\delta)= \sum_{\delta\in\mathcal V^i(\mathcal T_h)}\sum_{K\in\mathcal T_{\delta}}\sum_{K'\in\mathcal T_{\delta}}\frac{|K|}{N_{K'} |\omega_{\delta}| }\int_{K'}v~{\rm d}x \\
	&= \sum_{\delta\in\mathcal V^i(\mathcal T_h)}\sum_{K'\in\mathcal T_{\delta}}\frac{1}{N_{K'}}\int_{K'}v~{\rm d}x=\int_{\Omega}v~{\rm d}x.	
	\end{align*}
	On the other side, by the scaling argument,
	\[
	\|\mathcal Q_hv\|_{0,K}\lesssim h_K^{d/2}\sum_{\delta\in\mathcal V^i(K)}|(\mathcal Q_hv)(\delta)|\lesssim \sum_{\delta\in\mathcal V^i(K)}\|v\|_{0,\omega_{\delta}},
	\]
	which together with the inverse inequality produces \eqref{eq:QhL02prop}.
\end{proof}


Define $J_h: H_0^1(\Omega)\to Q_h^L\cap H_0^1(\Omega)$ by $J_hv:=I_h^{SZ}v+\mathcal Q_h(v-I_h^{SZ}v)$. Since $\int_{\Omega}J_hv~{\rm d}x=\int_{\Omega}v~{\rm d}x$, it holds that $J_h v \in Q_h$ when $v\in Q$.

\begin{lemma}
We have
\begin{equation}\label{eq:Jhestimate}
h^i|v-J_hv|_i\lesssim h^j|v|_j \quad\forall~v\in H_0^1(\Omega)\cap H^j(\Omega),\; i\leq j\leq 2,\; i=0,1.
\end{equation}
\end{lemma}
\begin{proof}
For each $K\in\mathcal T_h$, by \eqref{eq:QhL02prop} and $v-J_hv=v-I_h^{SZ}v-\mathcal Q_h(v-I_h^{SZ}v)$,
\begin{equation*}
h_K^i|v-J_hv|_{i,K}\lesssim h_K^i|v-I_h^{SZ}v|_{i,K} + \sum_{\delta\in\mathcal V^i(K)}\|v-I_h^{SZ}v\|_{0,\omega_{\delta}}.
\end{equation*}
Therefore, we acquire \eqref{eq:Jhestimate} from the estimate of $I^{SZ}_h$.
\end{proof}


To derive robust error estimates of $\|\bm u - \bm u_h\|_{V,h} + \|p- p_h\|_{Q}$ with respect to parameters $\iota$ and $\lambda$,
we introduce a finite element space
\begin{align*}
\tilde V_h:=&\{\bm v_h\in L^2(\Omega;\mathbb R^d)| \, \bm v_h|_K\in V(K) \textrm{ for each } K\in\mathcal T_h, \textrm{ all the  DoFs \eqref{H2lowestncfemnDdof1}-\eqref{H2lowestncfemnDdof4}} \\
&\qquad\qquad\qquad \textrm{ are single-valued}, \textrm{and DoFs \eqref{H2lowestncfemnDdof1}-\eqref{H2lowestncfemnDdof2} on boundary vanish}\}.
\end{align*}
Define an interpolation operator $\tilde I_h: H_0^1(\Omega; \mathbb R^d)\to \tilde V_h$ as follows:
\begin{align}
(\tilde I_h\bm v) (\delta)&=\frac{1}{\#\mathcal T_{\delta}}\sum_{K\in\mathcal T_{\delta}}(I_K^{BDM}\bm v)(\delta),\notag\\
(\tilde I_h\bm v, \bm q)_e&=\frac{1}{\#\mathcal T_{e}}\sum_{K\in\mathcal T_{e}}(I_K^{BDM}\bm v, \bm q)_e  \quad \forall~\bm q\in\mathbb P_{0}(e;\mathbb R^d), \notag\\
(\partial_{\bm{n}}(\Pi_F(\tilde I_h\bm v)), \bm q)_F&=\frac{1}{\#\mathcal T_{F}}\sum_{K\in\mathcal T_{F}}(\partial_{\bm{n}}(\Pi_F(I_K^{BDM}\bm v)), \bm q)_F  \quad \forall~\bm q\in\mathbb P_{0}(F;\mathbb R^{d-1}), \notag\\
(\div(\tilde I_h\bm v), q)_F&=\frac{1}{\#\mathcal T_{F}}\sum_{K\in\mathcal T_{F}}(\div(I_K^{BDM}\bm v), q)_F \quad \forall~q\in\mathbb P_{0}(F), \notag\\
(\tilde I_h\bm v, \bm q)_K&=(\bm v, \bm q)_K  \quad \forall~\bm q\in\mathbb P_{0}(K;\mathbb R^d), \notag
\end{align}
for $\delta\in \mathcal V^i(\mathcal T_h)$, $e\in \mathcal E^i(\mathcal T_h)$, $F\in \mathcal F(\mathcal T_h)$ and $ K\in \mathcal T_h$.


\begin{lemma}
	We have
	\begin{align}\label{eq:Ihtildeestimate2}
	|\tilde I_h\bm v|_{1}&\lesssim |\bm v|_1\quad\forall~\bm v\in H_0^1(\Omega;\mathbb R^d),
	\\
\label{eq:Ihtildeestimate3}
	\sum_{i=1}^2h_K^{i}|\bm v-\tilde I_h\bm v|_{i,K}  &\lesssim h_K^j\sum_{\delta\in\mathcal V(K)}|\bm v|_{j, \omega_{\delta}}\;\forall~\bm v\in H_0^1(\Omega;\mathbb R^d)\cap H^j(\Omega;\mathbb R^d), j=2,3.
	\end{align}
\end{lemma}
\begin{proof}
	Applying the similar argument as Lemma~\ref{lem:Ihestimate},
	we get for $i=1,2$ that
	\begin{align}
	&h_K^{2i}|\tilde I_h\bm v - I_K^{BDM}\bm v|_{i,K}^2\lesssim h_K\sum_{\delta\in \mathcal V(K)}\sum_{F\in \mathcal F(T_h),\delta\in F}\|\llbracket I_K^{BDM}\bm v\rrbracket\|_{0,F}^2 \notag\\
	&\qquad\qquad\quad +h_K^3\sum_{F\in \mathcal F^i(K)}\|\llbracket \partial_{\bm{n}}(I_K^{BDM}\bm v)\rrbracket\|_{0,F}^2+ h_K^{3}\sum_{F\in \mathcal F^i(K)}\|\llbracket Q_K(\div\bm v)\rrbracket\|_{0,F}^2. \label{eq:20220118}
	\end{align}
	Employing the inverse inequality, the trace inequality, the error estimates of $I_K^{BDM}$ and $Q_K$,
	\[
	|\tilde I_h\bm v - I_K^{BDM}\bm v|_{1,K}\lesssim \sum_{\delta\in\mathcal V(K)}|\bm v|_{1, \omega_{\delta}},
	\]
	which combined with the $H^1$ boundedness of $I_K^{BDM}$ yields \eqref{eq:Ihtildeestimate2}.
	
	On the other hand, by \eqref{eq:20220118}, we get
	\begin{align*}
	h_K^{2i}|\bm v-\tilde I_h\bm v|_{i,K}^2&\lesssim h_K^{2i}|\bm v-I_K^{BDM}\bm v|_{i,K}^2 + h_K\sum_{\delta\in \mathcal V(K)}\sum_{F\in \mathcal F(T_h),\delta\in F}\|\llbracket I_K^{BDM}\bm v\rrbracket\|_{0,F}^2 \\
	&\quad +h_K^3\sum_{F\in \mathcal F^i(K)}\|\llbracket \partial_{\bm{n}}(I_K^{BDM}\bm v)\rrbracket\|_{0,F}^2+ h_K^{3}\sum_{F\in \mathcal F^i(K)}\|\llbracket Q_K(\div\bm v)\rrbracket\|_{0,F}^2.
	\end{align*}
	Then \eqref{eq:Ihtildeestimate3} follows from the inverse inequality, the trace inequality, and the error estimates of $I_K^{BDM}$ and $Q_K$.
\end{proof}

\begin{lemma}
	We have
	\begin{equation}\label{eq:Ihtildeu}
	|\bm u-\tilde I_h\bm u|_1+\iota|\bm u-\tilde I_h\bm u|_{2,h}\lesssim h^{1/2}\|\bm f\|_0.
	\end{equation}
\end{lemma}
\begin{proof}
	By \eqref{eq:Ihtildeestimate2}-\eqref{eq:Ihtildeestimate3} and \eqref{eq:u0regularity}-\eqref{eq:straingradregularity1},
	\begin{align*}
	|\bm u-\bm u_0-\tilde I_h(\bm u-\bm u_0)|_1\lesssim h^{1/2}|\bm u-\bm u_0|_1^{1/2}|\bm u-\bm u_0|_2^{1/2}\lesssim h^{1/2}\|\bm f\|_0.
	\end{align*}
	And it follows from \eqref{eq:Ihtildeestimate3} and \eqref{eq:u0regularity} that
	\[
	|\bm u_0-\tilde I_h\bm u_0|_1\lesssim h|\bm u_0|_2\lesssim h\|\bm f\|_0.
	\]
	Combining the last two inequalities gives
	\[
	|\bm u-\tilde I_h\bm u|_1\lesssim h^{1/2}\|\bm f\|_0.
	\]
	According to \eqref{eq:Ihtildeestimate3} and \eqref{eq:straingradregularity1}, we have
	\[
	|\bm u-\tilde I_h\bm u|_{2,h}\lesssim h^{1/2}|\bm u|_2^{1/2}|\bm u|_3^{1/2}\lesssim \iota^{-1}h^{1/2}\|\bm f\|_0.
	\]
	Therefore, \eqref{eq:Ihtildeu} holds from the last two inequalities.
\end{proof}

\begin{lemma}
We have
\begin{align}\label{eq:Ihu}
|\bm u-I_h\bm u|_1+\iota|\bm u-I_h\bm u|_{2,h}&\lesssim h^{1/2}\|\bm f\|_0,
\\
\label{eq:Jhp}
\|p-J_hp\|_0+\iota|p-J_hp|_{1}&\lesssim h^{1/2}\|\bm f\|_0.
\end{align}
\end{lemma}
\begin{proof}
We only prove \eqref{eq:Ihu}, since the proof of \eqref{eq:Jhp} is similar.
By the definitions of $I_h$ and $\tilde I_h$, we get from the multiplicative trace inequality, the estimate of $I_h^{BDM}$ and \eqref{eq:u0regularity}-\eqref{eq:straingradregularity1} that
\begin{align*}
&|\tilde I_h\bm u-I_h\bm u|_1^2\lesssim \sum_{F\in\mathcal F^{\partial}(\mathcal T_h)}h_F\|\nabla(\bm u-I_h^{BDM}\bm u)\|_{0,F}^2 \\
\lesssim& \sum_{F\in\mathcal F^{\partial}(\mathcal T_h)}h_F\big(\|\nabla(\bm u-\bm u_0-I_h^{BDM}(\bm u-\bm u_0))\|_{0,F}^2 + \|\nabla(\bm u_0-I_h^{BDM}\bm u_0)\|_{0,F}^2\big)	\\
\lesssim& h|\bm u-\bm u_0|_{1}|\bm u-\bm u_0|_{2}+h^2|\bm u_0|_{2}^2\lesssim h\|\bm f\|_{0}^2.
\end{align*}
	Similarly, we have
	\begin{align*}
	|\tilde I_h\bm u-I_h\bm u|_{2,h}^2&\lesssim \sum_{F\in\mathcal F^{\partial}(\mathcal T_h)}h_F^{-1}\|\nabla(\bm u-I_h^{BDM}\bm u)\|_{0,F}^2 \lesssim h|\bm u|_{2}|\bm u|_{3}\lesssim \iota^{-2}h\|\bm f\|_{0}^2.
	\end{align*}
	Hence
\[
	|\tilde I_h\bm u-I_h\bm u|_1+\iota|\tilde I_h\bm u-I_h\bm u|_{2,h}\lesssim h^{1/2}\|\bm f\|_0.
\]
	Finally, \eqref{eq:Ihu} holds from \eqref{eq:Ihtildeu} and the last inequality.
\end{proof}

\subsection{Error estimates}
Similar to \cite{NilssenTaiWinther2001}, we have the following abstract error estimates.
\begin{lemma}\label{lm:errorup}
	Let $(\bm u,p)\in V\times Q$ and $(\bm u_h,p_h)\in V_h\times Q_h$ be the solution of problem~\eqref{straingradLameproblem} and problem \eqref{eq:discretestraingradmixedform}, respectively. Then
	\begin{equation*}
	\|\bm u - \bm u_h\|_{V,h} + \|p- p_h\|_{Q}\lesssim \inf_{\bm u_I\in V_h} \| \bm u - \bm u_I \|_{V,h} + \inf_{p_I\in Q_h}\|p - p_I\|_{Q} + E_h  ,
	\end{equation*}
	where
	\[
	E_h = \sup_{\bm v_h\in\bm V_h}\frac{ (\bm f,\bm v_h) - a_h(\bm u, \bm v_h) - b_h(\bm v_h ,p) }{\|\bm v_h\|_{V,h}}.
	\]
\end{lemma}
\begin{proof}
	Let $\tilde{\bm u}_h=\bm u_h-\bm u_I$ and $\tilde{p}_h=p_h-p_I$. It follows from the mixed finite element method \eqref{eq:discretestraingradmixedform} and the second equation of \eqref{eq:straingradmixedform} that
	\begin{align*}
	& a_h(\tilde{\bm u}_h, \bm v_h)+b_h(\bm v_h, \tilde{p}_h)-b_h(\tilde{\bm u}_h, q_h)+c(\tilde{p}_h,q_h) \\
	=&(\bm f, \bm v_h) - a_h(\bm u_I, \bm v_h)-b_h(\bm v_h, p_I)+b_h(\bm u_I, q_h)-c(p_I,q_h) \\
	=&(\bm f,\bm v_h) - a_h(\bm u, \bm v_h) - b_h(\bm v_h ,p) + a_h(\bm u-\bm u_I, \bm v_h)+b_h(\bm v_h, p-p_I)\\
	&-b_h(\bm u-\bm u_I, q_h)+c(p-p_I,q_h),
	\end{align*}
	which together with the discrete stability \eqref{eq:ncfemstability} ends the proof.
\end{proof}

\begin{lemma}
	We have
	\begin{align}
	E_h&\lesssim \iota h(|\bm u|_3+\lambda|\div\bm u|_2), \label{eq:Ehestimate1}
	\\
	E_h&\lesssim h^{1/2}\|\bm f\|_{0}. \label{eq:Ehestimate2}
	\end{align}
\end{lemma}
\begin{proof}
	For $\bm v_h\in V_h\subset H_0^1(\Omega;\mathbb R^d)$,
	apply integration by parts to the first equation in problem \eqref{eq:straingradproblem}
	to acquire
	\begin{align}
	&(\bm f,\bm v_h) - a_h(\bm u, \bm v_h) - b_h(\bm v_h ,p) \notag\\
	=& (\bm\sigma(\bm u), \bm \varepsilon(\bm v_h)) - \iota^2 (\Delta \bm \sigma(\bm  u), \bm \varepsilon(\bm v_h)) - a_h(\bm u, \bm v_h) - b_h(\bm v_h ,p) \notag\\
	=&- \iota^2 \sum\limits_{K\in \mathcal{T}_h}\int_{\partial K} \partial_{\bm n}\bm \sigma(\bm u): \bm\varepsilon
	(\bm v_h) {\rm d}S 
.  \notag
	\end{align}
	Let $Q_F^0$ be the standard $L^2$-projection from $L^2(F)$  to $\mathbb P_{0}(F)$, whose tensorial version is also denoted by $Q_F^0$.
	By the weak continuity \eqref{eq:weakcontinuity},
	\begin{align*}
	&(\bm f,\bm v_h) - a_h(\bm u, \bm v_h) - b_h(\bm v_h ,p)\\
	=&- \iota^2\sum_{K\in\mathcal T_h}\sum_{F\in \mathcal F(K)}\int_{F} (\partial_{\bm n}\bm\sigma(\bm u) - Q_F^0\partial_{\bm n}\bm\sigma(\bm u)):(\bm\varepsilon
	(\bm v_h) - Q_F^0\bm\varepsilon
	(\bm v_h)){\rm d} S.
	\end{align*}

	From the error estimate of $Q_F^0$ and the trace inequality, we acquire
	\[
	(\bm f,\bm v_h) - a_h(\bm u, \bm v_h) - b_h(\bm v_h ,p)\lesssim \iota^2h(|\bm u|_3+\lambda|\div\bm u|_2)|\bm v_h|_{2,h},
	\]
	\[
	(\bm f,\bm v_h) - a_h(\bm u, \bm v_h) - b_h(\bm v_h ,p)\lesssim \iota^2h^{1/2}(|\bm u|_2+\lambda|\div\bm u|_1)^{1/2}(|\bm u|_3+\lambda|\div\bm u|_2)^{1/2}|\bm v_h|_{2,h}.
	\]
	Hence
	\[
	E_h\lesssim \iota h(|\bm u|_3+\lambda|\div\bm u|_2),\quad E_h\lesssim \iota h^{1/2}(|\bm u|_2+\lambda|\div\bm u|_1)^{1/2}(|\bm u|_3+\lambda|\div\bm u|_2)^{1/2}.
	\]
	Therefore, we finish the proof by using \eqref{eq:straingradregularity1} and \eqref{eq:straingradregularity2}.
\end{proof}

\begin{theorem}\label{th:error}
	Let $(\bm u,p)\in (V\cap H^3(\Omega;\mathbb{R}^d)) \times (Q\cap H^2(\Omega))$ and $(\bm u_h,p_h)\in V_h\times Q_h$ be the solution of problem \eqref{straingradLameproblem} and \eqref{eq:discretestraingradmixedform}, respectively. 
	We have
	\begin{equation}\label{eq:errororder}
	\|\bm u - \bm u_h\|_{V,h} + \|p - p_h\|_{Q} \lesssim
	(h^2 + \iota h)|\bm u|_{3} + \lambda (h^2 + \iota h)|\div \bm u|_{2}.
	\end{equation}
\end{theorem}
\begin{proof}
	Let $\bm u_I = I_h \bm u\in V_h$, and $p_I=J_hp\in Q_h$ in Lemma \ref{lm:errorup}. Then by \eqref{eq:Ihestimate3} and \eqref{eq:Jhestimate},
	\begin{equation*}
	\|\bm u - \bm u_h\|_{V,h} + \|p- p_h\|_{Q}
	\lesssim (h^2 + \iota h )|\bm u|_{3} +\lambda(h^2 + \iota h)|\div \bm u|_{2} + E_h.
	\end{equation*}
	Therefore \eqref{eq:errororder} follows from \eqref{eq:Ehestimate1}.
\end{proof}

\begin{theorem}\label{tm:errorf}
With the same assumption of Theorem \ref{th:error}, we have
\begin{align}
\label{eq:errorf}
\|\bm u - \bm u_h\|_{V,h} +\|p - p_h\|_{Q}&\lesssim h^{1/2}\|\bm f\|_{0},
\\
\label{eq:errorfu0}
\|\bm u_0 - \bm u_h\|_{V,h} +\|p_0 - p_h\|_{Q}&\lesssim (\iota^{1/2}+h^{1/2})\|\bm f\|_{0}.
\end{align}
\end{theorem}
\begin{proof}
	The estimate \eqref{eq:errorf} follows from Lemma~\ref{lm:errorup}, \eqref{eq:Ihu}, \eqref{eq:Jhp} and \eqref{eq:Ehestimate2}. And estimate \eqref{eq:errorfu0} holds from \eqref{eq:errorf}, \eqref{eq:u0regularity}-\eqref{eq:straingradregularity1} and \eqref{eq:straingradregularity2}.
\end{proof}

\section{Numerical experiments}\label{Sect:test}

In this section, we present two numerical examples using the mixed FEM constructed in Sect \ref{Sect:H1ncFEM} toward testifying its uniform convergence and robustness with respect to $\lambda$ and $\iota$ on a uniform triangulation.


 To make the proposed method more accessible, we first present the mixed finite element spaces in two dimensions in detail. Let $\mathcal T_h$ be a shape regular triangulation.
The finite element spaces in two dimensions are
\begin{align*}
V_h=&\{\bm v_h\in H_0^1(\Omega;\mathbb R^2)| \, \bm v_h|_K\in V(K) \textrm{ for each } K\in\mathcal T_h, \textrm{ all the  DoFs \eqref{Ddof1}-\eqref{Ddof42}} \\
&\qquad\qquad\qquad \textrm{ are single-valued}, \textrm{and DoFs \eqref{Ddof1}-\eqref{Ddof31} on boundary vanish}\}, \\
Q_h =&\{q_h\in H^1_0(\Omega)\cap L^2_0(\Omega) \,|\, q_h|_K\in\mathbb P_1(K) \textrm{ for each } K\in\mathcal T_h\},
\end{align*}
where $V(K)=\mathbb P_{2}(K;\mathbb R^2)+b_K\mathbb P_{1}(K;\mathbb R^2)+b_K^2\mathbb P_{0}(K;\mathbb R^2)$.
The associated DoFs for $V_h$ are given as
\begin{align}
\bm v_h (\delta) & \quad\forall~\delta\in \mathcal V^i(\mathcal T_h),\label{Ddof1}\\
\bm v_h (m_e)& \quad \forall~  e\in \mathcal E^i(\mathcal T_h), \label{Ddof2}\\
\frac{1}{|e|}\int_e\partial_{\bm n}\bm v_h\,{\rm d}s & \quad \forall~  e\in \mathcal E^i(\mathcal T_h),  \label{Ddof31}\\
\frac{1}{|K|}\int_K \bm v_h\,{\rm d} x & \quad \forall~ K\in \mathcal T_h, \label{Ddof42}
\end{align}
where $m_e$ is the midpoint of $e$.
The DoFs for $Q_h$ are all the function values at interior vertices.
Note that the DoFs \eqref{Ddof1}-\eqref{Ddof31} on boundary vanish. The other point to be emphasized is that we here substitute DoFs \eqref{H2lowestncfemnDdof31}-\eqref{H2lowestncfemnDdof32} by the moments of the normal derivatives \eqref{Ddof31} owing to \eqref{eq:divdecomp} in order to simplify computation. Also, we use the function values of midpoints on edges \eqref{Ddof2} instead of \eqref{H2lowestncfemnDdof2}. With such substitution, the global finite element space remains the same.

%

\begin{figure}[t]
	\centering
\includegraphics[scale=0.5]{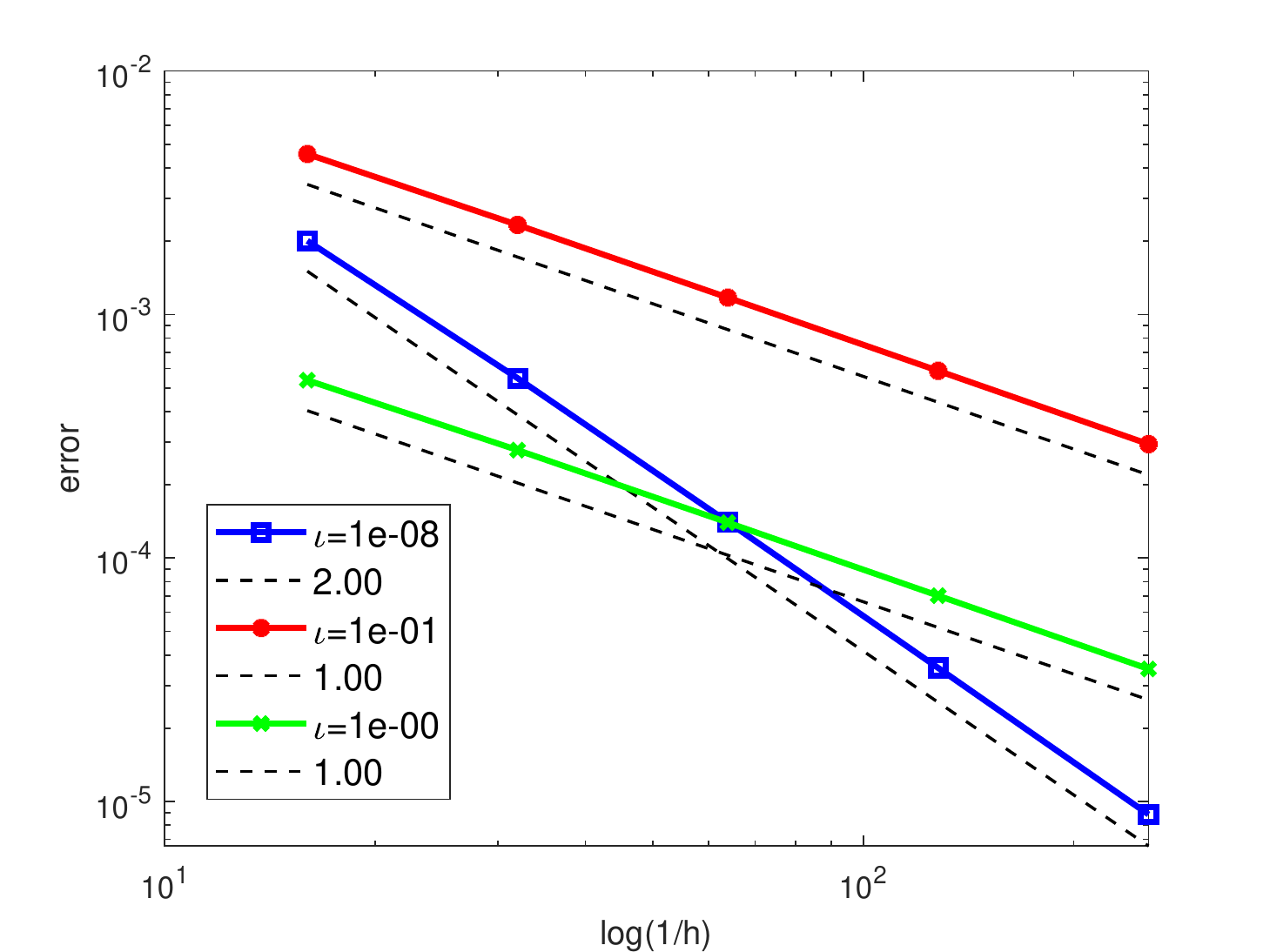}
	\caption{
		Convergence rate for $\lambda = 1e +08$.
	}
	\label{TableEx1}
\end{figure}

\begin{example}\label{example1}\rm
Let $\Omega = (0,1)^2$. We choose the right side function $\bm f$ such that the exact solution of SGE model \eqref{eq:straingradproblem} is
\begin{align*}
     \bm u=&
	\begin{bmatrix}
	 3(e^{\cos(2\pi x_1)}-e)^2\sin(2\pi x_2)\sin(\pi x_2)\\
      8(e^{2\cos(2\pi x_1)}-e^{1+\cos(2\pi x_1)})\sin(2\pi x_1)\sin^3(\pi x_2)
	\end{bmatrix}.
	\end{align*}
 Moreover, set $\mu=1$. It is easy to check that $\bm u$ is divergence-free without boundary layer, and hence $\bm f$ is independent of the Lam\'e constant $\lambda$. The purpose of this example is to demonstrate the error estimate given in Theorem \ref{th:error}.
\end{example}


We measure the relative error by
      \[
       E_u \coloneqq\frac {\|\bm u-\bm u_h\|_{V,h}} {\|\bm f\|_{0}}.
      \]
   The numerical results of $E_u$ are displayed in Table~\ref{tab:Eu} with different values of Lam\'e constant $\lambda$, size parameter $\iota$ and mesh size $h$. We may observe from Table \ref{tab:Eu} that
   when $\iota$ takes the value $1$ and $1e-01$, the relative error is linearly convergent and when $\iota=1e-08$, it is quadratically convergent. In particular, Fig. \ref{TableEx1} performs the relative error curves for different $\iota$ with fixed $\lambda=1e+08$. In this situation, if $\iota=1e-08$, $E_u$ behaves like $O(h^{2})$. If $\iota=1e-01$ or $1e-00$, $E_u=O(h)$. Hence, we may conclude that the convergence order of the error is quadratic as $\iota \ll h$. All these results are in good agreement with the error bound given in Theorem \ref{th:error}.

\begin{table}[t]
	\scriptsize
	\centering
	\caption{The performance of Example \ref{example1}.}
	\begin{tabular}{ccccccccccc}
		\toprule
		\multirow{2}{*}{$\lambda$}&	\multirow{2}{*}{$\iota$}  & \multicolumn{5}{c}{$h$} &\multirow{2}{*}{rate}  \\
		\cline{3-7}
		\multirow{2}{*}{} &\multirow{2}{*}{}& 1/16 & 1/32  & 1/64 & 1/128 & 1/256  \\
		\midrule
		\multirow{3}{*}{1e+00 }
        &$ 1e-00$& 5.375e-04 & 2.776e-04 & 1.399e-04 & 7.010e-05 & 3.507e-05 & 1.00 \\
        &$ 1e-01$& 4.561e-03 & 2.334e-03 & 1.173e-03 & 5.874e-04 & 2.938e-04 & 1.00 \\
        &$1e-08$ & 2.008e-03 & 5.477e-04 & 1.407e-04 & 3.540e-05 & 8.862e-06 & 2.00   \\
		\midrule
		\multirow{3}{*}{1e+04}
        &$ 1e-00$& 5.375e-04 & 2.776e-04 & 1.399e-04 & 7.010e-05 & 3.507e-05 & 1.00 \\
        &$ 1e-01$& 4.562e-03 & 2.334e-03 & 1.173e-03 & 5.874e-04 & 2.938e-04 & 1.00 \\
        &$1e-08$ & 2.009e-03 & 5.477e-04 & 1.407e-04 & 3.540e-05 & 8.862e-06 & 2.00  \\
		\midrule
		\multirow{3}{*}{1e+08}	
        &$ 1e-00$& 5.375e-04 & 2.776e-04 & 1.399e-04 & 7.010e-05 & 3.507e-05 & 1.00 \\
        &$ 1e-01$& 4.562e-03 & 2.334e-03 & 1.173e-03 & 5.874e-04 & 2.938e-04 & 1.00 \\
        &$1e-08$ & 2.009e-03 & 5.477e-04 & 1.407e-04 & 3.540e-05 & 8.862e-06 & 2.00   \\
		\bottomrule
	\end{tabular}
	\label{tab:Eu}
\end{table}
\begin{table}[t]
	\scriptsize
	\centering
	\caption{The performance of Example \ref{example_div-free}.}
	\begin{tabular}{ccccccccccc}
		\toprule
		\multirow{2}{*}{$\lambda$}&	\multirow{2}{*}{$\iota$}  & \multicolumn{5}{c}{$h$}  &\multirow{2}{*}{rate} \\
		\cline{3-7}
	 &\multirow{2}{*}{} & 1/16  & 1/32 & 1/64 & 1/128 & 1/256 \\
		\midrule
		\multirow{3}{*}{1e+00 }
        &$1e-04$ & 2.053e-02 & 1.447e-02 & 1.025e-02 & 7.340e-03 & 5.438e-03 & 0.43 \\
		&$1e-06$ & 2.052e-02 & 1.445e-02 & 1.020e-02 & 7.206e-03 & 5.093e-03 & 0.50 \\
        &$ 1e-08$& 2.052e-02 & 1.445e-02 & 1.020e-02 & 7.206e-03 &5.093e-03  & 0.50\\
		\midrule
		\multirow{3}{*}{1e+04}
        &$1e-04$ & 2.054e-02 & 1.447e-02 & 1.025e-02 & 7.341e-03 & 5.438e-03 & 0.43\\
		&$1e-06$ & 2.053e-02 & 1.446e-02 & 1.020e-02 & 7.207e-03 & 5.094e-03 & 0.50 \\
        &$ 1e-08$& 2.053e-02 & 1.446e-02 & 1.020e-02 & 7.207e-03 & 5.094e-03 & 0.50 \\
		\midrule
		\multirow{3}{*}{1e+08}	
        &$1e-04$ & 2.054e-02 & 1.447e-02 & 1.025e-02 & 7.341e-03 & 5.438e-03 & 0.43\\
		&$1e-06$ & 2.053e-02 & 1.446e-02 & 1.020e-02 & 7.207e-03 & 5.094e-03 & 0.50 \\
        &$ 1e-08$& 2.053e-02 & 1.446e-02 & 1.020e-02 & 7.207e-03 & 5.094e-03 & 0.50\\
		\bottomrule
	\end{tabular}
	\label{tab:divfreeEu0}
\end{table}
\begin{example}\label{example_div-free}\rm
	Let $\Omega = (0,1)^2$. The exact solution of the reduced problem \eqref{eq:elasticity} is set to be a divergence-free function in the form
    \begin{align*}
	\bm u_0=&
	\begin{bmatrix}
	-x_1^2(1-x_1)^2 x_2(1-x_2)(1-2x_2)\\
	x_1(1-x_1)(1-2x_1)x_2^2(1-x_2)^2
	\end{bmatrix}.
	\end{align*}
The right side term $\bm f$ computed from \eqref{eq:elasticity} is independent of both $\lambda$ and $\iota$. We use this $\bm f$ as the right side function of problem \eqref{eq:straingradproblem}. Since $\partial_{\bm n}\bm u_0|_{\partial \Omega}\not=\bm 0$, the function $\bm u$ has a strong boundary layer when $\iota$ is sufficient small. We still set $\mu=1$. We focus on investigating the robustness of our numerical method with respect to both $\lambda$ and $\iota$ in this example.

When $\iota \ll h$, it follows from  \eqref{eq:u0regularity}-\eqref{eq:straingradregularity1} and \eqref{eq:straingradregularity2} that
 \[
 \|\bm u-\bm u_0\|_{V}\lesssim \iota^{1/2}\|\bm f\|_0
                      \lesssim h^{1/2}\|\bm f\|_0,
 \]
 which combined with \eqref{eq:errorfu0} immediately implies \eqref{eq:errorf}. So we turn to verify the estimate \eqref{eq:errorfu0}, which depends on $\bm u_0$ instead of the unknown function $\bm u$.  Let $E_{u_0}=\frac{\|\bm u_h-\bm u_0\|_{V,h}}{\|\bm f\|_0}$. We compute its values for different values of Lam\'e constant $\lambda$ and the mesh size $h$ in Table~\ref{tab:divfreeEu0}. We may observe that $E_{u_0}=O(h^{0.43})$ with $\iota=1e-04$ and $E_{u_0}=O(h^{1/2})$ with $\iota=1e-06$ or $1e-08$ no matter what value $\lambda$ takes. We can see that the best convergence order is really $1/2$ when $\iota\ll h$, which is consistent with the estimate \eqref{eq:errorfu0} and shows the robustness of the mixed finite element method with respect to both the size parameter $\iota$ and Lam\'{e} coefficient $\lambda$.
\end{example}

\section*{Acknowledgments}
The second author would like to thank Prof. M. Dauge from Universit\'{e} de Rennes 1 for helpful discussion about elliptic boundary value problems in domains with corners and bringing his attention to the monograph \cite{KozlovMazyaRossmann2001spectral}. The authors are also indebted to the referees for valuable comments which improved an earlier version of the paper.

\bibliographystyle{abbrv}
\bibliography{./SGE-mixed}
\end{document}